\documentclass[11pt,a4paper,twoside]{amsart}
\usepackage[top=1in, bottom=1.25in, marginparwidth=1in, inner=1in, outer=1in]{geometry}
\usepackage[parfill]{parskip}
\usepackage{amssymb, amsthm, amsmath, blkarray, bm, mathrsfs,array,url,xcolor,comment}
\usepackage[all,cmtip]{xy}
\usepackage{graphicx}
\usepackage{tikz}
\usetikzlibrary{knots}
\usepackage{mathtools}
\usepackage{hyperref}
\usepackage{multirow}
\usepackage{array}
\newcolumntype{P}[1]{>{\centering\arraybackslash}p{#1}}
\usepackage{longtable}
\usepackage{caption} 
\usepackage{verbatim}

\makeatletter
\newtheorem*{rep@theorem}{\rep@title}
\newcommand{\newreptheorem}[2]{%
\newenvironment{rep#1}[1]{%
 \def\rep@title{#2 \ref{##1}}%
 \begin{rep@theorem}}%
 {\end{rep@theorem}}}
\makeatother

\newtheorem{theorem}{Theorem}
\newtheorem{lemma}[theorem]{Lemma}
\newtheorem{cor}[theorem]{Corollary}
\newtheorem{prop}[theorem]{Proposition}
\newtheorem{conj}[theorem]{Conjecture}
\newtheorem*{repeattheorem}{Theorem \ref{thm2adjknots}}
\theoremstyle{definition}
\newtheorem{mydef}[theorem]{Definition}

\theoremstyle{remark}

\numberwithin{theorem}{section}

\title{Constructing and Cataloging 2-Adjacent Knots}

\author{John Carney}
\address{Department of Mathematics \& Applied Mathematics, Virginia Commonwealth University, 1015 Floyd Avenue, Box 842014, Richmond, VA 23284-2014, USA}
\email{carneyjr@vcu.edu}

\author{Everett Meike}
\address{Department of Mathematics, North Carolina State University, Raleigh, NC 27607, USA}
\email{emeike@ncsu.edu}

\date{}

\begin{document}

\begin{abstract}
Generalizing unknotting number, $n$-adjacent knots have $n$ crossings such that changing any non-empty subset of them results in the unknot. In this paper, we determine the $2$-adjacent knots through $12$ crossings. Using Heegaard Floer $d$-invariants and the Alexander polynomial, we develop a new technique to obstruct $2$-adjacency, and we prove conjectures of Ito \cite{Ito} and Kato \cite{GoKato} regarding $2$-adjacent knots.     
\end{abstract}

\maketitle

\section{Introduction}

    A knot $K$ that is $\bm{n}$\textbf{-adjacent} to another knot $K'$ has a diagram with $n$ crossings such that changing any non-empty subset of them results in a diagram of $K'$. For simplicity, we say that $K$ is $n$-adjacent if $K'$ is the unknot, as in \cite{Tao1}.

In \cite{AK}, Askitas and Kalfagianni studied $n$-adjacency for $n>2$. They found that all $n$-adjacent knots have trivial Conway/Alexander polynomials when $n>2$. They also produced a construction method that describes all $n$-adjacent knots for $n>2$. Neither result holds for $2$-adjacent knots. In this paper, we explore how and why $2$-adjacent knots differ from other $n$-adjacent knots and completely catalog $2$-adjacent knots with $12$ or fewer crossings. 

Askitas and Kalfagianni showed that for $n>2$, there are no non-trivial alternating or fibered $n$-adjacent knots. This is not true for $2$-adjacent knots. The two smallest non-trivial knots, $3_1$ and $4_1$, are both $2$-adjacent, alternating, and fibered. For $n$ greater than 2, $n$-adjacent knots have the property that all their Vassiliev invariants of degree less than $2n-1$ vanish \cite{AK}. The same is not true for $2$-adjacent knots. 

Much of the previously published research on $2$-adjacent knots was completed by Tao \cite{Tao1}, who found restrictions on the Conway, Jones, and HOMFLY-PT polynomials of $2$-adjacent knots. In unpublished work, Ito's list of $2$-adjacent knots with 12 or fewer crossings \cite{Ito} coincides with ours, and Kato \cite{GoKato} tried to exclude all other knots from the list. We have independently verified and corrected elements from their papers to create a complete list of $2$-adjacent knots with $12$ pr less crossings.

\begin{theorem} \label{thm2adjknots}
    The following knots are $2$-adjacent: $3_1$, $4_1$, $8_{17}$, $8_{21}$, $9_{44}$, $10_{88}$, $10_{136}$, $10_{156}$, $11a_{289}$, $11n_{84}$, $11n_{125}$, $12a_{1008}$, $12a_{1249}$, $12n_{275}$, $12n_{392}$, $12n_{464}$, $12n_{482}$, $12n_{483}$, $12n_{650}$, and $12n_{831}$. No other knots with $12$ crossings or less are $2$-adjacent.
\end{theorem}

We use existing obstructions to rule out most of the non-$2$-adjacent knots with up to 12 crossings. We also prove Theorem \ref{thm-newobst} and Corollary \ref{cor-mccoy} as new obstructions to $2$-adjacency. However, these were not sufficient for five knots: $11a_{255},12a_{358},12n_{620}, 12n_{656},$ and $12n_{586}$. Building on results of Baker and Motegi \cite{bakermotegi}, we developed a new method combining the Montesinos trick with the Alexander polynomial and Heegaard Floer $d$-invariants to exclude $11a_{255},12a_{358},12n_{620}$, $12n_{656}$, and $12n_{586}$ from being $2$-adjacent. We prove the following obstruction in Section \ref{torres}, with notation defined in Section \ref{sec:dbc-montesinos}.

\begin{theorem}\label{thm-newobst}
    If $K$ is a $2$-adjacent knot, then $\det(K)=4\omega^2\pm1$ for some $\omega\in\mathbb{Z}$. If the crossings in the $2$-adjacency set are of opposite sign, then $\det(K)=4\omega^2+1$, and if they are of the same sign, then $\det(K)=4\omega^2-1$.
    Furthermore, let $\gamma_K$ be a crossing arc for $c$ in the $2$-adjacency set of $K$, and $\gamma_{\mathcal{U}}$ be the image of $\gamma_K$ after performing a crossing change at $c$. If $\omega\neq0$, and $J$ is the lift of $\gamma_{\mathcal{U}}$ in $\Sigma(\mathcal{U})=S^3$, then $\Delta_J(z)=1$ for all $z=e^{\frac{(2\ell+1)\pi i}{\omega}}$, where $\ell\in \mathbb{Z}$.
\end{theorem}

In Section \ref{sectionbackground}, we present some background information that will be necessary to prove our results. We prove Theorem \ref{thm-newobst}, a new obstruction to $2$-adjacency, using Ozsv\'{a}th and Szab\'{o}'s \textit{correction terms} (\textit{$d$-invariants}) \cite{absolutelygraded} and the double cover of $S^3$ branched over a knot in Section \ref{torres}. We then apply Theorem \ref{thm-newobst} to $11a_{255},12a_{358},12n_{620}$, and $12n_{656}$ in Section \ref{11a255results}, and to $12n_{586}$ in Section \ref{sec-12n586}. In Section \ref{sectionseifert} we describe a construction of $2$-adjacent knots adapted from the construction used by Askitas and Kalfagianni in \cite{AK}. An appendix is included with an explicit construction of all $2$-adjacent knots up to $12$ crossings.

\section{Background} \label{sectionbackground}

Some notation should be set out in advance. We use many common knot invariants and use standard notation as follows. 

\subsection{The Alexander polynomial}

The Alexander polynomial $\Delta_K(t)$ of a knot $K$ can be defined $\Delta_K(t)=\det(V-tV^T)$ up to multiplication by a unit, where $V$ is the Seifert matrix corresponding to a Seifert surface of $K$. We use the notation $-K$ to represent the mirror of a knot $K$.

Basic properties of the Alexander polynomial:

\begin{itemize}
    \item The Alexander polynomial is symmetric: $\Delta_K\left(t^{-1}\right)=\Delta_K\left(t\right)$ for all knots $K$.
    \item $\Delta_K(1)= \pm 1$.
    \item If $-K$ is the mirror of $K$, then $\Delta_{-K}(t)=\Delta_K(t)$.
    \item The determinant of a knot $K$ is $\det(K)=\Delta_K(-1)=|H_1(\Sigma(K))|$ (see Section \ref{sec:dbc-montesinos}).
    \item The Alexander polynomial is equivalent to the Conway polynomial of a knot, $\nabla_K(z)=\sum_{i=0}^na_iz^i$ under the relation $\Delta_K(t)=\nabla_K(t^\frac12-t^{-\frac12})$.
\end{itemize}

\subsection{The HOMFLY-PT polynomial}

The HOMFLY-PT polynomial of a knot is a Laurent polynomial in two variables, $\ell$ and $m$. Call it $P_K(\ell,m)=\sum_{i=0}^np_{i_K}(\ell)m^i$. Each $p_{i_K}(\ell)$ is a Laurent polynomial in $\ell$. We will use the following relations to define the HOMFLY-PT polynomial:

\begin{enumerate}
    \item $P_\mathcal{U}(\ell,m)=1$.
    \item $\ell P_{L_+}(\ell,m)+\ell^{-1}P_{L_-}(\ell,m)+mP_{L_0}(\ell,m)=0$.
\end{enumerate}

\subsection{Signed unknotting number} 

A knot $K$ has \textit{positive} (respectively, \textit{negative}) unknotting number $1$ if there exists a diagram with a positive (respectively, negative) crossing such that changing it produces a diagram of the unknot. A $2$-adjacency set can include unknotting crossings of the same sign or of different signs. Both display different properties.

\subsection{Known invariants of 2-adjacent knots}

We know the following basic facts about various invariants of $2$-adjacent knots:

\begin{prop}\label{prop-basic-obstructions}
If a knot $K$ is $2$-adjacent, then
    \begin{enumerate}
    \item The unknotting number of $K$ is $1$.
    \item The signature of $K$, $\sigma (K)$, is either $0$ or $\pm2$. If $K$ has positive unknotting number $1$ and negative unknotting number $1$, then $\sigma (K)=0$.
    \item If $K$ is rational, then $K$ is either $3_1$ or $4_1$.
    \item In the Conway polynomial of $K$, either $a_2=\pm1$ or $a_2=0$. In the case where $a_2=0$, $a_4$ is a perfect square.
    \end{enumerate}
\end{prop}

\begin{proof}
\hspace{.1in}
    \begin{enumerate}
        \item By definition.
        \item The unknot has signature $0$. Since changing a crossing in a diagram of a knot $K$ changes the signature by at most $2$ \cite{murasugi1965certain}, $\sigma(K)=\pm2$ or $\sigma(K)=0$. Furthermore, changing a positive (negative) crossing will increase (decrease) the signature by $0$ or $2$. Suppose a knot $J$ has both a positive and a negative unknotting crossing, and $\sigma(J)= 2$. Then, changing a positive crossing could not yield a knot with signature $0$. Similarly, if $\sigma(J)=-2$, then changing a negative crossing could not yield a knot with signature $0$. Therefore $\sigma(J)=0$.
        \item By Theorem 1.1 in \cite{TORISU_2004}.
        \item By Theorems 3.1 and 3.2 in \cite{Tao1}.
    \end{enumerate}
\end{proof}

\subsection{Dehn Surgery} \label{sec:ch2-dehnsurgery}

Dehn surgery modifies a $3$-manifold $Y$ by removing a tubular neighborhood of a knot (or link) and replacing it with a solid torus (or a disjoint union of tori) in a specified way. Here we follow the notation of Saveliev \cite{saveliev}. Let $K$ be a knot in a closed, oriented $3$-manifold $Y$. A tubular neighborhood $N(K)$ of a knot $K\subset Y$ is a solid torus $D^2\times S^1$. If we cut open $Y$ along the torus boundary $\partial N(K)$ of this solid torus, we get two manifolds: $N(K)$ and $Y\setminus \operatorname{int} \left(N(K)\right)$, the knot exterior. Note that $Y\setminus \operatorname{int} N(K)$ is a manifold with torus boundary, and $Y=Y\setminus \operatorname{int} N(K) \cup (D^2\times S^1)$. We can glue $D^2\times S^1$ back into the knot exterior via a homeomorphism $h: \partial D^2\times S^1 \rightarrow \partial (Y\setminus \operatorname{int} N(K))$ to obtain a closed orientable $3$-manifold $Y'=Y\setminus \operatorname{int} N(K) \cup_h (D^2\times S^1)$. This process is called \textit{Dehn surgery}. 

We call a curve $\partial D^2\times \{*\}$ a \textit{meridian} $\mu$, and it bounds a disk in $N(K)$. If we look at the image $h(\mu)$ on $\partial Y\setminus \operatorname{int} N(K)$, this completely determines our new manifold $Y'$. If $Y=S^3$, then up to isotopy, the curve $h(\mu)$ (or any curve on $\partial (S^3\setminus \operatorname{int} N(K))$) is given by integers $p$ and $q$, which are relatively prime. We can see this by taking note of two homology classes of curves in $N(K)=D^2\times S^1$. One is represented by a meridian $\mu$, and the other is the \textit{longitude} $\lambda$, represented by a simple curve which is nullhomologous in the knot complement and intersects the meridian transversely in exactly one point. These curves provide a basis for the first homology group $H_1(\partial (S^3\setminus \operatorname{int} N(K)))$. Thus, we can describe any simple closed curve $\gamma$ in $\partial (S^3\setminus \operatorname{int} N(K))$ as a linear combination of $\mu$ and $\lambda$. This curve $\gamma$, in turn, specifies a homeomorphism $h: T'\rightarrow T$ by mapping the meridian of $T'$ to (a curve isotopic to) $\gamma$. When performing this surgery where $h(\mu)=\gamma$, we call $\gamma$ the \textit{surgery slope}. If $[\gamma]=p\cdot [\mu] + q\cdot [\lambda]$, we say we have performed $\frac{p}{q}$-surgery on $K\subset S^3$ and $\frac{p}{q}\in\mathbb{Q}\cup \infty$ is called the \textit{surgery coefficient}.  We denote $\frac{p}{q}$-surgery on $K$ in $S^3$ by $S^3_{\frac{p}{q}}(K)$. This process naturally generalizes to surgery on an $n$-component link $L=L_1\sqcup \dots \sqcup L_n$ where the boundary of a tubular neighborhood of $L$ is $T_1\sqcup \dots \sqcup T_n$, surgery slopes $\gamma_1,\dots,\gamma_n$, and surgery coefficients $\frac{p_1}{q_1}, \dots, \frac{p_n}{q_n}$.

\subsection{The linking matrix} \label{sec:ch2-linkingform}

First we define \textit{linking number} of a $2$-component link.

\begin{mydef} 
    Let $L_1$ and $L_2$ be two disjoint oriented knots in $S^3$. Consider all crossings in a regular projection of $L_1\cup L_2$ such that $L_1$ crosses under $L_2$. Their \emph{linking number} $\operatorname{lk}(L_1,L_2)$ is defined as (\# positive such crossings) $-$ (\# negative such crossings).\
\end{mydef}

This allows us to define the \textit{linking matrix} of an $n$-component link.

\begin{mydef}  
    Let $\mathcal{L} = L_1\cup L_2 \cup \dots  \cup L_n$ be an oriented framed link in $S^3$, the $i^\text{th}$ component having framing $e_i\in \mathbb{Z}$. The matrix $A=a_{ij},\ i,j=1,\dots ,n,$ with the entries
    $$a_{ij} = \begin{cases}
        e_i, \ \ \ \ \ \ \ \  \ \ \ \ \ \ \text{if } i=j\\
        \operatorname{lk}(L_i,L_j), \ \ \ \ \text{if } i\not = j,
    \end{cases}
    $$
    is called the \emph{linking matrix} of $\mathcal{L}$.
\end{mydef}

\subsection{Branched double-covers and the Montesinos trick}\label{sec:dbc-montesinos}

A fundamental $3$-manifold associated to a knot $K$ is the \textit{double cover of $S^3$ branched along $K$}, denoted $\Sigma(K)$. We often refer to this manifold as the \textit{branched double-cover of $K$}. Intuitively, we can think of this as taking the double cover of $S^3\setminus K$ and then gluing $K$ back in to this new $3$-manifold. For example, the branched double-cover of the unknot is $S^3$, and the branched double-cover of a rational knot is a lens space. For an oriented manifold $M$, we use the notation $-M$ to represent $M$ with opposite orientation.

A key ingredient in this work will be the application of the Montesinos trick \cite{Montesinos1975SURGERYOL}, which relates the branched double-cover $\Sigma(K)$ of a knot $K$ to surgery $S^3_{\frac{p}{q}}(J)$ on another knot $J$. If $u(K)=1$, then there is a diagram that contains a crossing $c$ such that performing a crossing change at $c$ unknots $K$. In $S^3$, we can find a \textit{crossing disk} for $c$, i.e., a disk which intersects $K$ twice with zero algebraic intersection number. Let $D$ be the crossing disk associated with the unknotting crossing $c$. We can draw an unknotted arc $\gamma$ in $D$ with boundary on $K$ (as in the middle diagram in the top row of Figure \ref{fig.11a255montesinos}). We call $\gamma$ a \textit{crossing arc} for $c$. Note that modifying $K$ in a tubular neighborhood of $\gamma$ yields the unknot: a full twist of the two strands of $K$ inside a neighborhood of the framed arc $\gamma$ changes the over-crossing to an under-crossing, i.e. it produces a crossing change. From a dual perspective, we can visualize $K$ as an unknot together with a knotted arc $\zeta$ (as in the bottom left diagram of Figure \ref{fig.11a255montesinos}) such that modifying the unknot in a neighborhood of $\zeta$ gets us back to $K$. Thus, in $\Sigma(\mathcal{U}) = S^3$, $\zeta$ lifts to a knot $J$. The specific version of the Montesinos trick that we will rely on is a refinement of Proposition 4.1 in \cite{jgreenedonaldson} and Lemma 1.5 in \cite{OzsvSz-UK}, using this notation:

\begin{prop} \label{signedmontesinosthm}
     Suppose that $K$ is a knot with unknotting number one, and reflect it if necessary so that it can be unknotted by changing a negative crossing $c$ to a positive one. Let $\gamma$ be a crossing arc for $c$. Changing the crossing $c$, we get a diagram of the unknot together with an arc $\zeta$. Then $\Sigma(K) = S^3_{-\epsilon \frac{d}{2}}(J)$, where $J \subset \Sigma(\mathcal{U})=S^3$ is the lift of $\zeta$, $\epsilon = (-1)^{\frac{\sigma(K)}{2}}$, and $d = \det(K)$.
\end{prop} 

One helpful observation is that the branched double-cover of an unknotting number one knot has cyclic first homology.

\subsection{Heegaard Floer homology} \label{sec:ch2-hfhomology}

 Heegaard Floer homology, introduced by Ozsv\'ath and Szab\'o \cite{ozsvath2004holomorphic}, is a rich set of topological invariants for a closed, oriented $3$-manifold $Y$ equipped with a $\operatorname{Spin}^c$ structure. We refer the reader to \cite{ozsvath2004holomorphic} for a thorough discussion, but we will discuss some necessary terms here. Heegaard Floer homology assigns a sequence of bigraded modules, for instance $\widehat{HF}(Y)$, to a $3$-manifold $Y$, and can provide a combinatorial way to encode their topology via chain complexes. Specifically, $\widehat{HF}(Y) $ is typically computed as a sum of vector spaces $\widehat{HF}(Y, \mathfrak{t})$, where $\mathfrak{t}$ runs over the set of $\operatorname{Spin}^c$ structures on $Y$. Thus, the Heegaard Floer homology of a $3$-manifold $Y$ can be decomposed into submodules, each corresponding to a distinct $\operatorname{Spin}^c$ structure. We will need the following definition in Section \ref{rulingoutalexpolys}.

    \begin{mydef}
        Let $M$ be a rational homology three-sphere. If $\operatorname{rk} \widehat{HF}(M)=\left|H_1(M ; \mathbb{Z})\right|$, then $M$ is an $L$-space \cite{ozsvath2004holomorphic}. If a knot $K\subset S^3$ has the property that $S^3_{\frac{p}{q}}(K)$ is an $L$-space for some $\frac{p}{q}>0$, then $K$ is called an $L$-space knot.
    \end{mydef}

\textit{Knot Floer homology}, a version of Heegard Floer homology for knots and links \cite{ozsvath2003holomorphicdisksknotinvariants}\cite{rasmussen2003floerhomologyknotcomplements}, can detect the genus, fiberedness (whether a knot complement fibers over a circle), as well as the unknot. Knot Floer homology also categorifies the Alexander polynomial.

    In Section \ref{rulingoutalexpolys}, we will explain that when $\Sigma(K)$ is an $L$-space, we can find all possible Alexander polynomials for a knot $J$ with $S^3_{\pm \frac{d}{2}}(J)=\Sigma(K)$ using Heegaard Floer \textit{$d$-invariants}.

    \subsubsection{Heegaard Floer $d$-invariants}

        For a rational homology $3$-sphere $Y$ equipped with a $\operatorname{Spin}^c$ structure $\mathfrak{t}$, Ozsv\'{a}th and Szab\'{o} define ``correction terms'' ($d$-invariants) $\in\mathbb{Q}$ \cite{absolutelygraded}.
        
        \begin{mydef} Let $Y$ be a rational homology three-sphere. The correction term $d(Y, \mathfrak{t})$ is the minimal grading, $\widetilde{\mathrm{gr}}$, of any non-torsion element in the image of $H F^{\infty}(Y, \mathfrak{t})$ in $H F^{+}(Y, \mathfrak{t})$.
        \end{mydef}
        
        \noindent If we let $\overline{\mathfrak{t}}$ denote the conjugate of $\mathfrak{t}$, these correction terms have the following properties:
            \[
            d(Y, \mathfrak{t})=d(Y, \overline{\mathfrak{t}})
            \]
        and
            \[
            d(Y, \mathfrak{t})=-d(-Y, \mathfrak{t})
            \]

        The $d$-invariants of a lens space $L(p,q)$ are calculated directly using a canonical ordering (see Proposition 4.8 in Section 4.1 of \cite{absolutelygraded}) of the $\operatorname{Spin}^c$ structures, indexed by $i$:

    \begin{prop}\cite{absolutelygraded}\label{dinvt-fla}
        Fix positive, relatively prime integers $p>q$, and also choose an integer with $0 \leq i<p+q$. Then we have the following recursive formula:
            $$
            d(-L(p, q), i)=\left(\frac{p q-(2 i+1-p-q)^2}{4 p q}\right)-d(-L(q, r), j),
            $$
            where $r$ and $j$ are the reductions modulo $q$ of $p$ and $i$, respectively.
    \end{prop}

\section{Torres/Baker-Motegi obstruction} \label{torres}

In this section, we prove Theorem \ref{thm-newobst}, which is a novel method for obstructing $2$-adjacency. Theorem \ref{thm-newobst} together with the discussion in Section \ref{rulingoutalexpolys} will be used to obstruct  $11a_{255}$, $12a_{358}$, $12n_{586}$, $12n_{620}$, and $12n_{656}$ from being $2$-adjacent.  To begin, we will outline several necessary prerequisite results. 

Let $K$ be a knot in $S^3$ with unknotting number equal to one. Recall from Section \ref{sec:dbc-montesinos} that there is a diagram that contains a crossing $c$ such that performing a crossing change at $c$ unknots $K$. In $S^3$, we can find a crossing disk $D$ for $c$, an unknotted crossing arc $\gamma$ in $D$ with boundary on $K$, and after performing a crossing change at $c$, we have an unknot together with an arc $\zeta$ which lifts to a knot $J$ in $\Sigma(\mathcal{U})=S^3$. 

We will also use the following relation between the two-variable Alexander polynomial\\$\Delta_{L_1\cup L_2}(x,y)\in\mathbb{Z}[x^{\pm1},y^{\pm1}]$ of an oriented link $L_1 \cup L_2$ (with $\operatorname{lk}\left(L_1, L_2\right)=\omega$) and the Alexander polynomial $\Delta_K(t)\in\mathbb{Z}[t^{\pm1}]$ of one component due to Torres \cite{torres}:
\begin{equation}\label{torresformula}
\Delta_{L_1 \cup L_2}(t, 1) \doteq \frac{t^\omega-1}{t-1} \Delta_{L_1}(t)
\end{equation}
Here, $\doteq$ signifies equivalence up to multiplication by a unit in the corresponding Laurent polynomial ring.

Baker and Motegi \cite{bakermotegi} apply \eqref{torresformula} to a special family of links  $\kappa_n \cup c_n$ as in Figure \ref{fig:bakermotegi}. They define $\kappa_n\cup c_n$ to be the family of links obtained by $\left(-\frac{1}{n}\right)$-surgery on a disjoint unknot $c$, where $c$ does not bound a disk disjoint from $\kappa$, and it is not a meridian of $\kappa$. In this case, they call the sequence of knots that are the images of $\kappa$ after $\left(-\frac{1}{n}\right)$-surgery on $c$ a \textit{twist family of knots $\{\kappa_n\}$ obtained by twisting the knot $\kappa$ along $c$}. 

\begin{figure}

    \centering
\begin{minipage}[b]{0.2\textwidth}
		\centering
		\resizebox{\linewidth}{!}{
    \begin{tikzpicture}[line width=5pt]
 \begin{knot}[
 	           clip width=20,
 	           clip radius=5pt,
	            consider self intersections=no splits, 
	            end tolerance=0.1pt
 	           ]

\strand[thick] (0,0) to[out=90,in=270] (0,6);
\strand[thick] (4,0) to[out=90,in=270] (4,6);
\strand[thick] (0,6) to[out=90,in=90,looseness=1.2] (4,6);
\strand[thick] (0,0) to[out=270,in=270,looseness=1.2] (4,0);

\strand[thick] (1,0) to[out=90,in=270] (1,6);
\strand[thick] (3,0) to[out=90,in=270] (3,6);
\strand[thick] (1,6) to[out=90,in=90,looseness=1.2] (3,6);
\strand[thick] (1,0) to[out=270,in=270,looseness=1.2] (3,0);

\draw[dashed,line width=1pt] (0,3.5) -- (.9,3.5);
\draw[dashed,line width=1pt] (0,5.7) -- (.9,5.7);

\draw[dashed,line width=1pt] (3,3.5) -- (4,3.5);
\draw[dashed,line width=1pt] (3,4.7) -- (4,4.7);

\strand[green!40!gray,line width=2pt] (4.25,4.5) 
  arc[start angle=60,end angle=-240,
      x radius=1.5, y radius=0.35];

\flipcrossings{1,2}
\end{knot}

\draw[blue,line width=2pt,fill=white] (2.0,2.0) rectangle (5,2.9);

\draw[black,line width=2pt,fill=white] (-.25,4.0) rectangle (1.25,5.1);

\node at (0.5,4.55) {\Large \textbf{T}};

\node at (0.25,7.5) {\Large $K_n$};

\node at (3.54,2.45) {\Large \textcolor{blue}{n twists}};

\node at (4.75,4.75) {\Large \textcolor{green!40!gray}{\textbf{c}}};
\end{tikzpicture}
		}
	\end{minipage}
    \caption{\textit{Twisting a knot $\kappa$ along an unknot $c$.} This is a simplified diagram. We assume nontrivial linking between $c$ and $\kappa$, so the $n$ twists are applied to all strands passing through $c$, producing $\kappa_n$.}
    \label{fig:bakermotegi}
\end{figure}
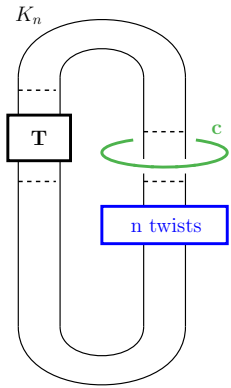

\begin{prop}[\cite{bakermotegi}]\label{prop:bakermotegi}
    Let $\{\kappa_n\}$ be the twist family of knots in a homology sphere obtained by twisting the knot $\kappa$ along an unknot $c$. Then
    $$
    \frac{t^\omega-1}{t-1} \Delta_{\kappa_n}(t) \doteq \Delta_{\kappa \cup c}\left(t, t^{n \omega}\right) .
    $$
\end{prop}

We now have the necessary results to begin our proof of Theorem \ref{thm-newobst}, restated here:

\begingroup
\def\thetheorem{\ref{thm-newobst}}
\begin{theorem}
    If $K$ is a $2$-adjacent knot, then $\det(K)=4\omega^2\pm1$ for some $\omega\in\mathbb{Z}$. If the crossings in the $2$-adjacency set are of opposite sign, then $\det(K)=4\omega^2+1$, and if they are of the same sign, then $\det(K)=4\omega^2-1$.
    Furthermore, let $\gamma$ be a crossing arc for $c$ in the $2$-adjacency set of $K$, and $\zeta$ be the image of $\gamma$ after performing a crossing change at $c$. If $\omega\neq 0$, and $J$ is the lift of $\zeta$ in $\Sigma(\mathcal{U})=S^3$, then $\Delta_J(z)=1$ for all $z=e^{\frac{(2\ell+1)\pi i}{\omega}}$, where $\ell\in \mathbb{Z}$.
\end{theorem}
\endgroup

\begin{proof}
Suppose an oriented knot $K$ in $S^3$ is $2$-adjacent via crossings $c_1, c_2$. For each crossing $c_i$, we find a disjoint crossing arc $\gamma_i$. After changing crossing $c_i$, we have a diagram of the unknot together with two disjoint arcs $\zeta^i_1,\zeta^i_2$. If we change both crossings at the same time, we have another diagram of the unknot with arcs $\eta_1,\eta_2$ corresponding to $\gamma_1,\gamma_2$ in our original diagram of $K$.

\[
\xymatrix{
 & \zeta^1_1, \zeta^1_2  \ar[dr]^{c_2}& \\
 \gamma_1, \gamma_2  \ar[ur]^{c_1}  \ar[dr]_{c_2}& & \eta_1, \eta_2 \\
& \zeta^2_1, \zeta^2_2  \ar[ur]_{c_1} &
}
\]

In particular, $\eta_1 \cup \eta_2$ lift to a link $L_1\cup L_2$ in $\Sigma(\mathcal{U}) \cong S^3$ that admits a $(\frac{p_1}{q_1}, \frac{p_2}{q_2})$--surgery to $\Sigma(K)$, and $\zeta^i_i$ lifts to a knot which is the image of $L_i$ after surgery along $L_{j}, j\not= i$. We will refer to the lift of $\zeta^1_1$ as $J_1$, and the lift of $\zeta^2_2$ as $J_2$. The branched double cover $\Sigma(K)$ can then be obtained by $\frac{r_i}{s_i}$--surgery along either $J_i$. By Proposition \ref{signedmontesinosthm}, $\frac{r_i}{s_i} \in \{ + \frac{d}{2}, - \frac{d}{2}\}$, where $d=\operatorname{det}(K)$, and $q_i=2$. Since both $c_1$ and $c_2$ are unknotting crossings, after changing either one, we have an unknot with branched double cover $S^3$.

\[
\xymatrix{
 & \Sigma(\mathcal{U})  \cong  S^3 \supset J_1 \ar[dl]_{\hspace{-1.75cm} \frac{r_1}{2}-\text{surgery on } J_1} & \\
\Sigma(K) & & \Sigma(\mathcal{U}) \cong S^3 \supset L_1\cup L_2 \ar[ul]_{\hspace{1cm} \frac{p_2}{2}-\text{surgery on } L_2}  \ar[dl]^{\hspace{1cm} \frac{p_1}{2}-\text{surgery on } L_1} \\
  & \Sigma(\mathcal{U})  \cong  S^3 \supset J_2  \ar[ul]^{\hspace{-1.25cm} \frac{r_2}{2}-\text{surgery on } J_2} &
}
\]
Going from right to left in the diagram above, we can see that $\frac{p_i}{2}$-surgery on $J_i$ takes us from $S^3$ to $S^3$. By a classical result of Gordon and Luecke \cite{Gordon1989KnotsAD}, no non-trivial surgery on a non-trivial knot in $S^3$ yields $S^3$, and thus we must have performed surgery on an unknot (which has determinant $d'=1$). Therefore by Proposition \ref{signedmontesinosthm},  $\frac{p_1}{q_1}, \frac{p_2}{q_2} \in \{ + \frac{1}{2}, - \frac{1}{2}\}$.

Proposition \ref{signedmontesinosthm} implies that the sign of the surgery coefficient for $J_i$ is determined by the sign of the crossing $c_i$ and the signature of $K$. In particular:
\begin{itemize}
	\item If $c_1, c_2$ have the same sign, Proposition \ref{signedmontesinosthm} implies that $\frac{r_1}{2}, \frac{r_2}{2}$ have the same sign.
	\item If $c_1, c_2$ have different signs, Proposition \ref{signedmontesinosthm} implies that $\frac{r_1}{2}, \frac{r_2}{2}$ have different signs.
\end{itemize}
 
If $\frac{r_1}{2}, \frac{r_2}{2}$ have the same sign, then $\frac{p_1}{2}, \frac{p_2}{2}$ must also have the same sign (otherwise one obtains different orientations on $S^3$). If $\frac{r_1}{2}, \frac{r_2}{2}$ have different signs, $\frac{p_1}{2}, \frac{p_2}{2}$ must have different signs. 

Thus if $c_1$, $c_2$ have the same sign, $\frac{p_1}{2},  \frac{p_2}{2}$ are the same sign, and if $c_1$, $c_2$ have different signs, $\frac{p_1}{2},  \frac{p_2}{2}$ have different signs.

The determinant of $K$, $\operatorname{det}(K)$, is also the determinant of the linking matrix for $L_1\cup L_2$. Using this, and the fact that $\frac{p_i}{2} \in \{ +\frac{1}{2}, -\frac{1}{2} \}$, we have
    \[
        \operatorname{det}(K) = \operatorname{det}\left(\left[\begin{array}{cc}
        \pm 1 & 2\omega \\
        2\omega & - 1
        \end{array}\right]\right),
    \]
 where (up to mirroring of $K$), the first entry is $-1$ if the signs of $c_1, c_2$ are the same and $+1$ if the signs of $c_1, c_2$ are different.

Since the Alexander polynomial of each component $L_i$ is trivial, \eqref{torresformula} gives us:
\begin{equation}\label{ourtorres}
    \Delta_{L_1 \cup L_2}(t, 1) \doteq \frac{t^\omega-1}{t-1}
\end{equation}

Now we can apply Proposition \ref{prop:bakermotegi} to our $L_1\cup L_2$, where $\kappa=L_1$ and $c=L_2$. We set $n=2$, and thus $\kappa_2$ is the image of $L_1$ after $(-\frac{1}{2})$-surgery on $L_2$ (and likewise, $\kappa_2\cup c_2$ is the image of $L_1\cup L_2$ after $(-\frac{1}{2})$-surgery on $L_2$). Then Proposition $\ref{prop:bakermotegi}$ together with \eqref{ourtorres} tell us that

\begin{equation}\label{torresobst}
\frac{t^\omega-1}{t-1} \Delta_{\kappa_2}(t)  \doteq \Delta_{L_1 \cup L_2}\left(t, t^{2 \omega}\right) .
\end{equation}

Our obstruction relies on the following:

When $t$ is a $(2\omega)^\text{th}$ root of unity $z=e^{2k\pi i/2\omega}$ where $\omega\not=0$ and $k$ is odd, then $z^\omega -1\not= 0$ and $z-1 \not= 0$, and we have
\begin{align*}
    \frac{z^\omega-1}{z-1}\Delta_{\kappa_2}(z) &= \Delta_{L_1 \cup L_2}(z,z^{2\omega}) \\
    &= \Delta_{L_1 \cup L_2}(z,1) \\
    &= \Delta_{\mathcal{U}}(z,1)\\
    &= \frac{z^\omega-1}{z-1} \Delta_\mathcal{U}(z) \\
    &= \frac{z^\omega-1}{z-1}.
\end{align*}

Therefore, for any $z=e^{((2\ell+1)\pi i)/\omega}$ and $\ell\in\mathbb{Z}$, $\Delta_{\kappa_2}(z)=1$.

\end{proof}

As noted above, we will use Theorem \ref{thm-newobst} in order to obstruct $11a_{255}$, $12a_{358}$, $12n_{620}$, and $12n_{656}$ from being $2$-adjacent. These knots are all known to have unknotting number equal to one, so for each knot, we can find a diagram with at least one unknotting crossing. This allows us to explicitly find a $J$ (and thus $\Delta_J(t))$ such that $S^3_{\pm \frac{d}{2}}(J)=\Sigma(K)$. For an example, see Figure \ref{fig.11a255montesinos}. Applying Theorem \ref{thm-newobst} to $\Delta_J(t)$ does not immediately obstruct $2$-adjacency, since the corresponding unknotting crossing is not necessarily a member of a $2$-adjacency set. However, we do know that a lift $J'$ of an arc $\zeta$ coming from an unknotting crossing in a $2$-adjacency set must have the property that $S^3_{\pm \frac{d}{2}}(J')=\Sigma(K)$. In Section \ref{rulingoutalexpolys}, we will show how to leverage an explicit $J$ (up to mirroring) to find all possible Alexander polynomials $\Delta_{J'}(t)$ such that $S^3_{\pm \frac{d}{2}}(J')=\Sigma(K)$. Then, applying Theorem \ref{thm-newobst} to all such $\Delta_{J'}(t)$ can rule out $2$-adjacency for a specific knot.

The unknotting number of $12n_{586}$ is unknown, however we found a rational tangle replacement to the unknot. In Section \ref{sec-12n586}, we use a more general version of the Montesinos trick to find the $d$-invariants for $\Sigma(12n_{586})$, and then work backwards to find possible Alexander polynomials for a lift $J'$, under the assumption that $\Sigma(12n_{586})$ is obtained by half-integral surgery on $J'$. Once we have found these Alexander polynomials, we can apply Theorem \ref{thm-newobst}.

\section{Ruling out \texorpdfstring{$\Delta_J(t)$}{DJ} via \texorpdfstring{$d$}{d}-invariants when \texorpdfstring{$\Sigma(K)$}{Sigma(K)} is an \texorpdfstring{$L$}{L}-space}\label{rulingoutalexpolys}

In this section, for a knot $K$ with certain properties, we discuss how to find all possible Alexander polynomials $\Delta_{J'}(t)$  of a knot $J'$ such that $S^3_{\frac{d}{2}}(J')=\Sigma(K)$. In particular, we use this in conjunction with Theorem \ref{thm-newobst} to obstruct $2$-adjacency for the knots $11a_{255}$, $12a_{358}$, $12n_{620}$, and $12n_{656}$ in order to prove Theorem \ref{thm2adjknots}. First, we find the lift $J$ of an unknotting arc in $K$, and use the Alexander polynomial $\Delta_J(t)$ together with $d$-invariants from Heegaard Floer homology to find other possible $\Delta_{J'}(t)$ such that $S^3_{\frac{d}{2}}(J')=\Sigma(K)$. Once we achieve this, we can apply Theorem \ref{thm-newobst} to all such $\Delta_{J'}(t)$ in order to obstruct $2$-adjacency. In Section \ref{sectionbackground} we defined Heegaard Floer $L$-spaces. In order to obstruct $11a_{255}, 12a_{358}, 12n_{620}$, $12n_{656}$, and $12n_{586}$ from being $2$-adjacent, we rely on the fact that their double-branched covers are $L$-spaces.

\begin{lemma}
    For $K\in \{ 11a_{255}, 12a_{358}, 12n_{620}, 12n_{656}, 12n_{586}\},\  \Sigma(K)$ is an $L$-space.
\end{lemma}

\begin{proof}
     Since these knots all have thin Khovanov homology (computed using the KnotTheory package in Mathematica), their double branched covers are $L$-spaces \cite{ozsvath-hfh-dbc}.
\end{proof}

If $\Sigma(K)$ is an $L$-space, we can find all possible Alexander polynomials for a knot $J$ with $S^3_{\pm \frac{d}{2}}(J)=\Sigma(K)$ using Heegaard Floer \textit{$d$-invariants}.

    \subsection{The Ni-Wu formula}\label{sec:niwufla}

    Ni and Wu proved a very useful formula for bounding the $d$-invariants of a manifold obtained by rational surgery on a knot $J$ in $S^3$ by relating them to $d$-invariants of lens spaces, with $\operatorname{Spin}^c$ structures indexed by $i$ as in Proposition \ref{dinvt-fla}.
    
        \begin{prop}[The Ni-Wu formula \cite{niwu}]\label{niwueq}
            Suppose $p,q>0$, and fix $0\leq i\leq p-1$. Then there is an ordering of the $\operatorname{Spin}^c$ structures $\mathfrak{t}_0,\mathfrak{t_1},\dots$ such that
            $$
            d\left(S_\frac{p}{q}^3(J), \mathfrak{t}_i\right)=d(L(p, q), i)-2 \max \left\{V_{\left\lfloor\frac{i}{q}\right\rfloor}, V_{\left\lfloor\frac{p+q-1-i}{q}\right\rfloor}\right\}.
            $$
        \end{prop}

    The reader may refer to \cite{niwu} for an in-depth explanation of $V_i$, but for our purposes, we only need to know their relation to the Alexander polynomial of an $L$-space knot, and the facts that $V_i\geq V_{i+1}\geq V_i-1$ and $V_i\geq0$. We will use the Ni-Wu formula to find the $d$-invariants of $\Sigma(K)=S_{\frac{d}{2}}^3(J)$. We can assume a positive $\frac{d}{2}$ surgery (as required) up to mirroring.
    First, we can simplify the formula in Proposition \ref{dinvt-fla} to the relevant case, where $p=\det(K)=d$ and $q=2$. We have:

        \begin{equation} \label{lensspace-dinvts}
            d(L(d,2),i) = -\frac{2 d - (2 i + 1 - d - 2)^2}{8 d} - \frac{(-1)^i}{4}.
        \end{equation}
    
    Recall that for conjugate $\operatorname{Spin}^c$ structures $\mathfrak{t}, \overline{\mathfrak{t}}$, $d(Y,\mathfrak{t})=d(Y,\overline{\mathfrak{t}})$. Ni and Wu reference Ozsv\'ath and Szab\'o \cite{hfkandrational} to define their ordering on $\operatorname{Spin}^c$ structures, which in turn is the same ordering given in \cite{OzsvSz-UK}. Conjugate $\operatorname{Spin}^c$ structures in the lens space correspond to conjugate $\operatorname{Spin}^c$ structures in the surgered manifold. Also note that the determinant $d$ of a knot is odd. From the proof of Theorem 4.1 in \cite{OzsvSz-UK}, we can see that the $\operatorname{Spin}^c$ structure corresponding to $i$ is conjugate to the $\operatorname{Spin}^c$ structure corresponding to $d+1-i$ for $1\leq i\leq d$ (and therefore, calculating $d$-invariants for $i=1,\dots,\frac{d+1}{2}$ is sufficient). We can see that these $d$-invariants align as well:
    $$
    \begin{aligned}
        d(L(d,2),i) 
        &= -\frac{2 d - (2 i + 1 - d - 2)^2}{8 d} - \frac{(-1)^i}{4} \\
        &= -\frac{2 d - (2 (d+1-i) + 1 - d - 2)^2}{8 d} - \frac{(-1)^{d+1-i}}{4} \\
       &= d(L(d,2), d+1-i) .
    \end{aligned}
    $$     
        
    Next, we need to calculate the $V_i$'s using $\Delta_J(t)$. Since $\Sigma(K)$ is an $L$-space and can be described as surgery on a knot $J$, we call $J$ an $L$-space knot. When $J$ is an $L$-space knot, by \cite{hfkandrational},
        
        \begin{equation}\label{ti}
            V_i = \sum_{j\geq 1} ja_{i+j} ,
        \end{equation}
        
        \noindent where $a_i$ is the coefficient of $\Delta_J(t)$ when it is expressed in its symmetrized form
        $$
        \Delta_J(t) = a_0+\sum_{i=1}^g a_i(t^i+t^{-i}),
        $$ 
        where $g$ is the genus of $J$.
    
    For our purposes, we take a knot $K$ with unknotting number $u(K)=1$ and $\Sigma(K)$ an $L$-space. First we found the lift $J$ of a crossing arc in a diagram of the unknot, which when changed gets us back to $K$ as in Figure \ref{fig.11a255montesinos}. Then we uploaded the image to KnotFolio, found the DT code, and confirmed that $S^3_{\frac{d}{2}}(J)=\Sigma(K)$ in SnapPy. We then calculated $\Delta_J(t)$ in Sage. Table \ref{tab:results} shows these results for $11a_{255}$, $12a_{358}$, $12n_{620}$, and $12n_{656}$ using their respective determinants and specific unknotting arc used in our calculations. Using the Alexander polynomial of the lift $J$ of an unknotting crossing arc, we can compute the $V_i$'s and then use the Ni-Wu formula to obtain the $d$-invariants of $S^3_{ \frac{d}{2}}(J)$:
    
        \begin{equation}\label{niwueq2}
            d\left(S_{\frac{d}{2}}^3(J), \mathfrak{t}_i\right)=d(L(d,2), i)-2 V_{\left\lfloor\frac{i}{2}\right\rfloor}.
        \end{equation}
   
    Note that if any $d$-invariants of $S^3_{\frac{d}{2}}(J)=\Sigma(K)$ are congruent modulo $2$, then it is possible that there could be a different ordering on the $\operatorname{Spin}^c$ structures for a different surgery description of $\Sigma(K)$, thus altering some $V_i$ and producing a different Alexander polynomial. (This will be shown explicitly in Section \ref{11a255results}).

\setlength{\tabcolsep}{3pt}

\renewcommand{\arraystretch}{0.85}

\begin{longtable}{|>{\centering\arraybackslash}p{0.08\linewidth}|>{\centering\arraybackslash}p{0.4\linewidth}|>{\centering\arraybackslash}p{0.3\linewidth}|>{\centering\arraybackslash}p{0.12\linewidth}|>{\centering\arraybackslash}p{0.05\linewidth}|}

    \caption{The first column shows our potential $2$-adjacent knot $K$, the second column gives the DT code of the lifted arc $J$, the third column gives the Alexander polynomial of $J$, and the last column evaluates the Alexander polynomial at a $2\omega^\text{th}$ root of unity, confirming that $K$ is not $2$-adjacent.} \label{tab:results} \\
    \hline 
    {\scriptsize $K$} & {\scriptsize {\scriptsize DT Code for $J$}} & {\scriptsize $\Delta_J(t)$} & {\scriptsize $\Delta_J\left(e^{\frac{2\pi i}{2\omega}}\right)$} \\ \hline
    
    \hline
    {\scriptsize $11a_{255}$} & {\scriptsize [$20$, $-98$, $-76$, $54$, $-60$, $-104$, $36$, $62$, $86$, $-50$, $-72$, $96$, $2$, $-100$, $-4$, $78$, $56$, $12$, $106$, $110$, $16$, $-18$, $-112$, $90$, $68$, $-26$, $28$, $-80$, $34$, $102$, $-84$, $-108$, $-40$, $44$, $-116$, $94$, $-24$, $30$, $-6$, $32$, $-58$, $10$, $-42$, $-66$, $46$, $114$, $70$, $22$, $74$, $52$, $-8$, $-82$, $-14$, $38$, $-64$, $88$, $-48$, $92$]} & {\scriptsize $t^{-27} - t^{-26} + t^{-22} - t^{-21} + t^{-17} - t^{-16} + t^{-13} - t^{-12} + t^{-11} - t^{-10} + t^{-8} - t^{-7} + t^{-6} - t^{-5} + t^{-3} - t^{-2} + t^{-1} - 1 + t - t^2 + t^3 - t^5 + t^6 - t^7 + t^8 - t^{10} + t^{11} - t^{12} + t^{13} - t^{16} + t^{17} - t^{21} + t^{22} - t^{26} + t^{27}$} & {\scriptsize $-0.1+0.5i$} \\ \hline 
    {\scriptsize $12a_{358}$} & {\scriptsize [$22$, $-144$, $-122$, $-60$, $-170$, $64$, $126$, $200$, $-210$, $180$, $-162$, $-142$, $-114$, $56$, $88$, $4$, $-228$, $-14$, $-128$, $152$, $182$, $-160$, $-106$, $74$, $-242$, $-140$, $-112$, $26$, $168$, $-120$, $-192$, $-172$, $204$, $-206$, $236$, $216$, $44$, $48$, $136$, $52$, $220$, $222$, $54$, $24$, $166$, $118$, $6$, $-226$, $-12$, $-66$, $154$, $184$, $-158$, $-46$, $138$, $80$, $86$, $-28$, $90$, $-58$, $92$, $-194$, $-174$, $202$, $-208$, $234$, $214$, $-244$, $-50$, $110$, $78$, $164$, $-196$, $-176$, $-16$, $-130$, $-68$, $186$, $218$, $72$, $134$, $-2$, $-30$, $-116$, $-190$, $-224$, $96$, $34$, $198$, $-212$, $-132$, $-70$, $-238$, $82$, $84$, $8$, $94$, $32$, $-230$, $-148$, $-36$, $-98$, $100$, $38$, $150$, $232$, $-20$, $-42$, $-104$, $-240$, $-188$, $10$, $62$, $124$, $146$, $-178$, $-18$, $-40$, $-102$, $-156$, $108$, $76$]} & {\scriptsize $t^{-53} - t^{-52} + t^{-44} - t^{-43} + t^{-39} - t^{-38} + t^{-35} - t^{-34} + t^{-30} - t^{-29} + t^{-26} - t^{-25} + t^{-24} - t^{-23} + t^{-21} - t^{-20} + t^{-17} - t^{-16} + t^{-15} - t^{-14} + t^{-12} - t^{-11} + t^{-10} - t^{-9} + t^{-8} - t^{-7} + t^{-6} - t^{-5} + t^{-3} - t^{-2} + t^{-1}- 1 + t - t^2 + t^3 - t^5 + t^6 - t^7 + t^8 - t^9 + t^{10} - t^{11} + t^{12} - t^{14} + t^{15} - t^{16} + t^{17} - t^{20} + t^{21} - t^{23} + t^{24} - t^{25} + t^{26} - t^{29} + t^{30} - t^{34} +t^{35} - t^{38} + t^{39} - t^{43} + t^{44} - t^{52} + t^{53}$} & {\scriptsize $-0.1+0.4i$} \\ \hline
    {\scriptsize $12n_{620}$} & {\scriptsize [$18$, $52$, $76$, $-102$, $80$, $108$, $58$, $-34$, $-110$, $46$, $72$, $-98$, $-2$, $8$, $104$, $106$, $-12$, $-84$, $-112$, $64$, $-90$, $68$, $-94$, $24$, $74$,$ -100$, $-26$, $10$, $32$, $14$, $86$, $114$, $40$, $92$, $44$, $96$, $22$, $50$, $6$, $-28$, $-54$, $16$, $60$, $-36$, $116$, $66$, $-42$, $70$, $-20$, $-48$, $-4$, $78$, $30$, $-56$, $-82$, $88$, $62$, $-38$]} & {\scriptsize $t^{-26} - t^{-25} + t^{-21} - t^{-20} + t^{-17} - t^{-16} + t^{-13} - t^{-12} +  t^{-10} - t^{-9} + t^{-8} - t^{-7} + t^{-5} - t^{-4} + t^{-3} - t^{-2} +t^{-1} - 1 + t - t^2 + t^3 - t^4 + t^5 - t^7 + t^8 - t^9 + t^{10} - t^{12} + t^{13} - t^{16} + t^{17} - t^{20} + t^{21} - t^{25} + t^{26}$} & {\scriptsize $-0.1+0.5i$} \\ \hline
    {\scriptsize $12n_{656}$} & {\scriptsize [$8$, $-58$, $40$, $24$, $72$, $-32$, $34$, $50$, $-52$, $54$, $-56$, $2$, $-62$, $-76$, $64$, $-48$, $-68$, $18$, $-20$, $4$, $60$, $26$, $-28$, $66$, $10$, $-70$, $-36$, $38$, $-22$, $-6$, $42$, $-78$, $74$, $12$, $-14$, $16$, $-30$, $-46$, $44$]} & {\scriptsize $t^{-17} - t^{-16} + t^{-13} - t^{-12} + t^{-9} - t^{-8} + t^{-6} - t^{-5} + t^{-9} -  t^{-3} + t^{-2} - t^{-1} + 1 - t + t^2 - t^3 + t^4 - t^5 + t^6 - t^8 +  t^9 - t^{12} + t^{13} - t^{16} + t^{17}$} & {\scriptsize $-0.2 + 0.6i$} \\
    \hline
\end{longtable}

    \subsection{Results} \label{final5knots}

    Here we will obstruct the knots $11a_{255}$, $12a_{358}$, $12n_{620}$, and $12n_{656}$ from being $2$-adjacent. We will rigorously go through the obstruction for $11a_{255}$, but the same process applies to the other three knots. In Figure \ref{fig.11a255montesinos}, we show the steps for lifting the labeled unknotting crossing arc for $11a_{255}$. For the remaining knots, we will exclude the intermediate steps. After finding the $d$-invariants for $\Sigma(K)$, we check for any permutations of $\operatorname{Spin}^c$ structures that could yield different Alexander polynomials for a knot $J'$ such that $\Sigma(K)=S^3_{\frac{d}{2}}(J')$. For each $K\in \{11a_{255}, 12a_{358}, 12n_{620}, 12n_{656}\}$, we found that the corresponding $\Delta_J(t)$ was the only viable Alexander polynomial. Therefore we were able to obstruct these knots from being $2$-adjacent by applying Theorem \ref{thm-newobst}.

        \begin{figure}
            \centering
            \includegraphics[scale=.25]{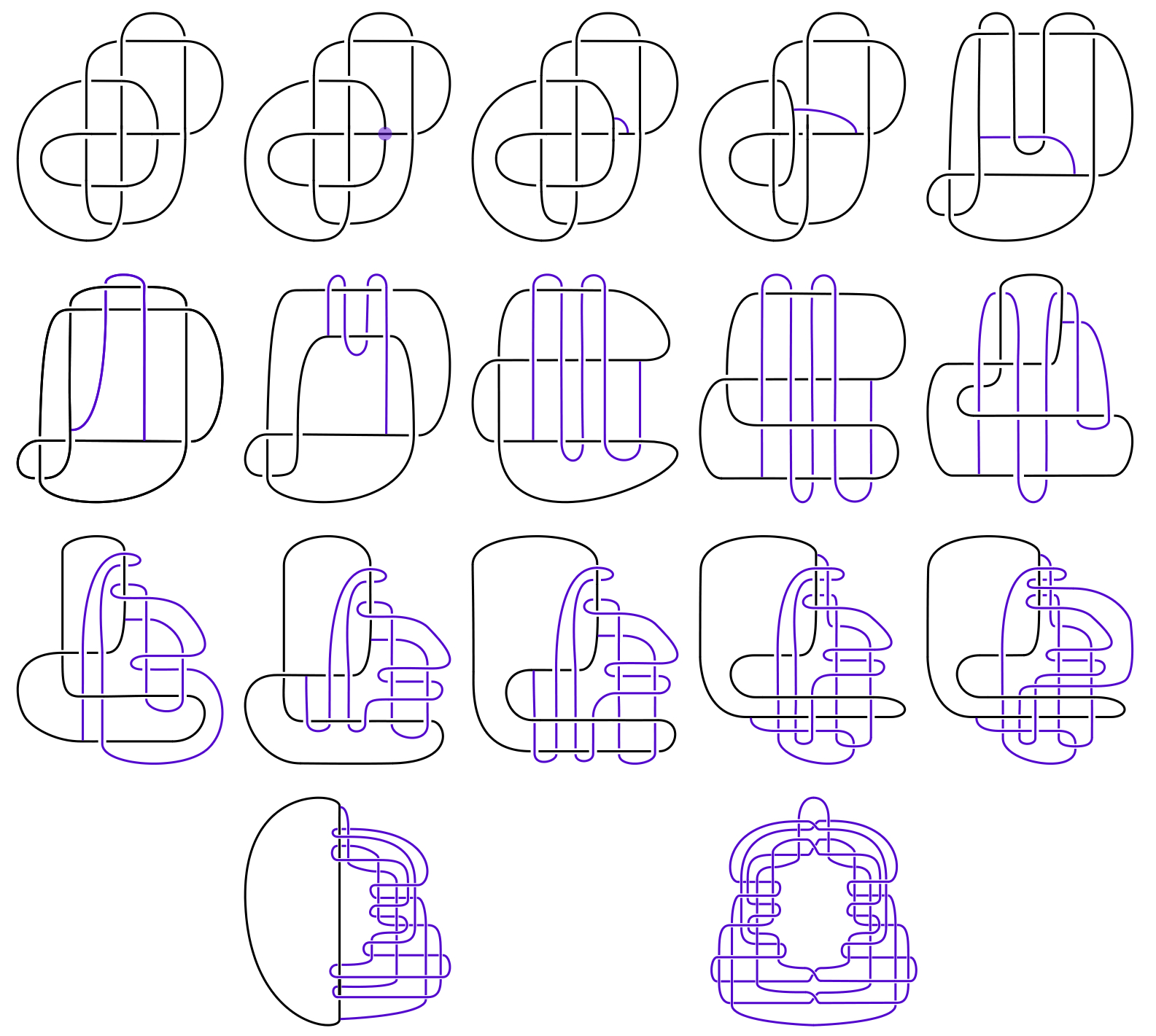}
            \caption{Finding the lift of an unknotting arc in $11a_{255}$. The top left figure is the original knot. To the right of that, we circle an unknotting crossing in purple, and in the next one, the marked crossing has been changed, recorded by a purple crossing arc. The following figures show the process of isotoping the black unknot until it looks like the standard unknot, while keeping track of the purple crossing arc until we reach the bottom left figure. The last image shows the lift $J$ of the arc in $\Sigma(K)$.}
            \label{fig.11a255montesinos}
        \end{figure}

\setlength{\tabcolsep}{5pt} 

\renewcommand{\arraystretch}{1.3} 
    
    \begin{table}[ht]
        \centering
        \begin{tabular}{|c|c|c|c|c|c|} \hline 
            Knot $K$ & $11a_{255}$ & $12a_{358}$  &  $12n_{620}$ & $12n_{656}$  & $12n_{586}$\\ \hline 
            Determinant & $143$ & $255$ & $143$ & $99$ & $101$\\ \hline 
            $\det(K)=4\omega^2 \pm 1$ &$4(6^2)-1$ & $4(8^2)-1$ & $4(6^2)-1$ & $4(5^2)-1$ & $4(5^2)+1$\\ \hline 
            Signature & 2 & -2 & -2 & 2 & 0 \\ \hline 
        \end{tabular}
        \caption{Determinant and signature for $11a_{255}$, $12a_{358}$, $12n_{620}$, $12n_{656}$, and $12n_{586}$. If $\det(K)=4\omega^2-1$, and $K$ is $2$-adjacent, then the crossings in the $2$-adjacency set have the same sign. If $\det(K)=4\omega^2+1$, then one crossing in the $2$-adjacency set is positive, and the other is negative.}
        \label{tab:my_label}
    \end{table}

    \subsection{Note on crossing signs}

        For $\kappa\in \{11a_{255}, 12n_{656}\}$, $\sigma(\kappa)=2$, so by Proposition \ref{basic-obst-knots}, and Proposition \ref{signedmontesinosthm}, the crossings in the $2$-adjacency set (if it exists), must be negative, and there must be a positive $\frac{\det(\kappa)}{2}$-surgery to $\Sigma(\kappa)$. 
    
        Similarly, for $\lambda\in \{12a_{358}, 12n_{656}\}$, $\sigma(\lambda)= -2$, so the $2$-adjacency set would have to contain two positive crossings, and $\Sigma(\lambda)=S^3_{-\frac{\det(\lambda)}{2}}(\lambda)$. In this case, $\Sigma(\lambda)= -S^3_{\frac{\det(\lambda)}{2}}(-\gamma)$, where $\gamma$ is the lift of the unknotting crossing arc. Therefore $\Sigma(-\lambda) = -\Sigma(\lambda) = S^3_{\frac{\det(\lambda)}{2}}(-\gamma)$. Since the Alexander polynomial of $\gamma$ is the same as that of $-\gamma$, we calculate the $V_i$'s and $d$-invariants for $\Sigma(-\lambda)$, and then find all possible Alexander polynomials of a knot $\gamma '$ such that $\Sigma(-\lambda) = S^3_{\frac{\det(\lambda)}{2}}(\gamma ')$. But then since the Alexander polynomial of $\gamma '$ is the same as that of $-\gamma '$, this calculation is sufficient.

        Therefore, our calculations obstruct $+\frac{\det(\kappa)}{2}$-surgery to $\Sigma(\kappa)$ when the crossings in the $2$-adjacency set must both be negative, and $+\frac{\det(\lambda)}{2}$-surgery to $\Sigma(-\lambda)$ when the crossings in the $2$-adjacency set would need to be positive.
    
        \subsection{\texorpdfstring{\bm{$11a_{255}$}}{11a255}} \label{11a255results}

        Note that the determinant of $11a_{255}$ is $143$, so we calculated the $d$-invariants $d(L(143,2),i)$ of the lens space $L(143,2)$:
        
        $\scriptstyle
            \left(\frac{5041}{286}, \frac{4757}{286}, \frac{4761}{286}, \frac{4481}{286}, \frac{4489}{286}, \frac{383}{26}, \frac{325}{22}, \frac{3953}{286}, \frac{3969}{286}, \frac{3701}{286}, \frac{3721}{286}, \frac{3457}{286}, \frac{3481}{286}, \frac{3221}{286}, \frac{3249}{286}, \frac{2993}{286}, \frac{275}{26}, \frac{2773}{286}, \frac{2809}{286}, \frac{197}{22}, \frac{2601}{286}, \frac{2357}{286}, \frac{2401}{286},\right.\\  \left. \quad \scriptstyle \frac{2161}{286}, \frac{2209}{286}, \frac{1973}{286}, \frac{2025}{286}, \frac{163}{26}, \frac{1849}{286}, \frac{1621}{286}, \frac{1681}{286}, \frac{1457}{286}, \frac{117}{22}, \frac{1301}{286}, \frac{1369}{286}, \frac{1153}{286}, \frac{1225}{286}, \frac{1013}{286}, \frac{99}{26}, \frac{881}{286}, \frac{961}{286}, \frac{757}{286}, \frac{841}{286}, \frac{641}{286}, \frac{729}{286}, \frac{41}{22}, \frac{625}{286}, \frac{433}{286},\right.\\  \left. \quad \scriptstyle \frac{529}{286}, \frac{31}{26}, \frac{441}{286}, \frac{257}{286}, \frac{361}{286}, \frac{181}{286}, \frac{289}{286}, \frac{113}{286}, \frac{225}{286}, \frac{53}{286}, \frac{13}{22}, \frac{1}{286}, \frac{11}{26}, -\frac{43}{286}, \frac{81}{286}, -\frac{79}{286}, \frac{49}{286}, -\frac{107}{286}, \frac{25}{286}, -\frac{127}{286}, \frac{9}{286}, -\frac{139}{286}, \frac{1}{286}, -\frac{1}{2}  \right)$
        
        In Figure \ref{fig.11a255montesinos}, we show the steps for lifting the labeled unknotting crossing arc for $11a_{255}$. The last image in the figure shows the lift $J$ of the arc in $\Sigma(11a_{255})$. Once we found $J$, we uploaded the image to KnotFolio, found the DT code, and confirmed that $S^3_{\frac{143}{2}}(J)=\Sigma(11a_{255})$ in SnapPy. We then calculated $\Delta_J(t)$ in Sage. When $\Delta_J(t)$ is expressed in its symmetrized form
        $$
        \Delta_J(t) = a_0+\sum_{i=1}^g a_i(t^i+t^{-i}),
        $$ 
        
        the symmetrized Alexander coefficients $(a_0,a_1,a_2,\dots)$ for the lifted arc $J$ are:
        
        {\small $(-1, 1, -1, 1, 0, -1, 1, -1, 1, 0, -1, 1, -1, 1, 0, 0, -1, 1, 0, 0, 0, -1, 1, 0, 0, 0, -1, 1, 0, 0, \dots)$ }
        
        Using \ref{ti}, we calculate the $V_i$'s:
        
        $(9, 8, 8, 7, 7, 7, 6, 6, 5, 5, 5, 4, 4, 3, 3, 3, 3, 2, 2, 2, 2, 2, 1, 1, 1, 1, 1, 0, 0, \dots )$

        Putting these together, we calculate the $d$-invariants of $\Sigma(11a_{255})=S^3_{\frac{143}{2}}(J)$ using the Ni-Wu formula for the relevant case (\ref{niwueq2}):

        $\scriptstyle \left( -\frac{107}{286}, \frac{181}{286}, \frac{185}{286}, -\frac{95}{286}, -\frac{87}{286}, \frac{19}{26}, \frac{17}{22}, -\frac{51}{286}, -\frac{35}{286}, -\frac{303}{286}, -\frac{283}{286}, \frac{25}{286}, \frac{49}{286}, -\frac{211}{286}, -\frac{183}{286}, \frac{133}{286}, \frac{15}{26}, -\frac{87}{286}, -\frac{51}{286}, -\frac{23}{22}, -\frac{259}{286}, \frac{69}{286}, \frac{113}{286}, \right. \\ \left. \quad  -\frac{127}{286}, -\frac{79}{286}, \frac{257}{286}, \frac{309}{286}, \frac{7}{26}, \frac{133}{286}, -\frac{95}{286}, -\frac{35}{286}, -\frac{259}{286}, -\frac{15}{22}, \frac{157}{286}, \frac{225}{286}, \frac{9}{286}, \frac{81}{286}, -\frac{131}{286}, -\frac{5}{26}, -\frac{263}{286}, -\frac{183}{286}, -\frac{387}{286},  -\frac{303}{286}, \frac{69}{286}, \frac{157}{286}, -\frac{3}{22}, \right. \\ \left. \quad  \frac{53}{286}, -\frac{139}{286}, -\frac{43}{286}, -\frac{21}{26}, -\frac{131}{286}, -\frac{315}{286}, -\frac{211}{286}, \frac{181}{286}, \frac{289}{286}, \frac{113}{286}, \frac{225}{286}, \frac{53}{286}, \frac{13}{22}, \frac{1}{286}, \frac{11}{26}, -\frac{43}{286},  \frac{81}{286}, -\frac{79}{286}, \frac{49}{286}, -\frac{107}{286}, \frac{25}{286}, -\frac{127}{286}, \frac{9}{286}, -\frac{139}{286}, \right. \\ \left. \quad  \frac{1}{286}, -\frac{1}{2}\right)$
        
       Now we can see explicitly that the lens space $d$-invariant values for $i=1,\dots, 71$ come in pairs that agree modulo $2$. The index of the self-conjugate $\operatorname{Spin}^c$ structure is $i=72$, and from our calculations we observe it is unique modulo $2$. For example,
        
       $$d(L(143,2),9)\equiv d(L(143,2),31) \equiv \frac{537}{286} \pmod{2}$$
       
       Thus 
       \begin{equation*}
            d(S^3_{\frac{143}{2}}(J),\mathfrak{t}_{9})=\frac{3969}{286}-2 V_{4}
        \end{equation*}
        and
        \begin{equation*}
           d(S^3_{\frac{143}{2}}(J),\mathfrak{t}_{31})=\frac{1681}{286}-2 V_{15}
        \end{equation*}

        Thus, $\Delta_J(t)$ tells us that $V_{4}= 7$ and $V_{15}= 3$, but there is the possibility of the existence of some knot $J'$ where  $\Sigma(K)=S^3_{\frac{143}{2}}(J')$ and $V_{4}$ and $V_{15}$ are switched in $\Delta_{J'}(t)$, resulting in a different Alexander polynomial, i.e. $\Delta_J(t)\not=\Delta_{J'}(t)$. More than two of the lens space $d$-invariants could agree modulo $2$, and the torsion invariants could be permuted, yielding a different Alexander polynomial. Therefore we calculated all of the possible sequences of the $V_i$ in Mathematica, and for all four sequences in question, we found the $V_i$ had to be unique using the facts that $V_i\geq V_{i+1}\geq V_i-1$, and $V_i\geq 0$ for every $i$. 

        Since $\det(11a_{255})=143=4(6^2)- 1$, we evaluate $\Delta_J(t)$ at an odd $12^\text{th}$ root of unity, shown in Table \ref{tab:results}. Since $\Delta_J(t)$ did not evaluate to $1$, by Theorem \ref{thm-newobst}, $11a_{255}$ is not $2$-adjacent.

        \subsection{\texorpdfstring{$\bm{12n_{586}}$}{12n586}}\label{sec-12n586}

            \begin{figure}
                \centering
                \includegraphics[width=0.8\linewidth]{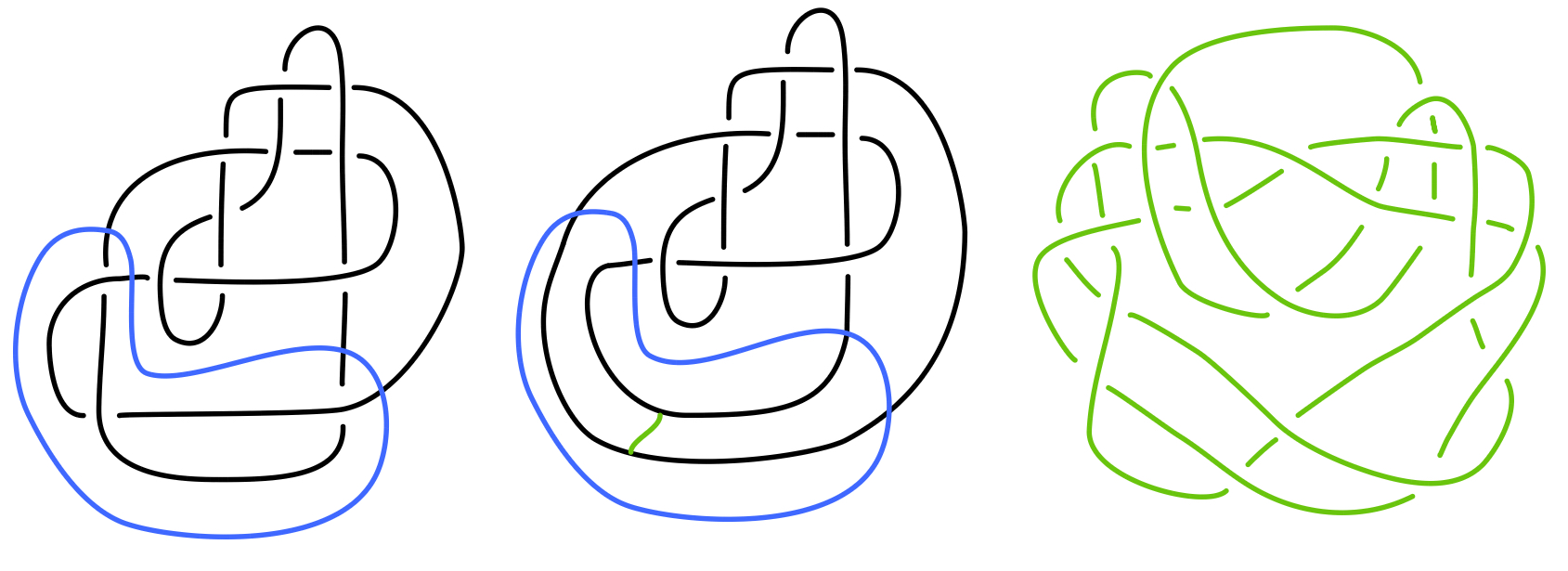}
                \caption{On the left, we see the knot $12n_{586}$ with a three half-twist tangle circled. In the middle diagram, we replace the circled tangle with an unknotted arc. On the right, we see the lift $J$ of the green arc from the middle diagram.}
                \label{fig:12n586tanglereplacement}
            \end{figure}
            
            For this particular knot, the unknotting number is unknown. According to KnotInfo \cite{knotinfo}, it is either $1$ or $2$. However, since its determinant is $101$, and therefore of the form $4\omega^2+1$, we know that if the knot is $2$-adjacent, the $2$-adjacency set of unknotting crossings must include crossings of both signs. Thus we can assume there exists a knot $J$ such that $S^3_{\frac{101}{2}}(J) \cong \Sigma(K)$ and do not need to consider a knot $J'$ such that $S^3_{-\frac{101}{2}}(J) \cong \Sigma(K)$. Obstructing positive $\frac{101}{2}$-surgery to $\Sigma(12n_{586})$ is sufficient. Since we cannot find a diagram with a single unknotting crossing, the previous strategy falls short. Instead, we used a more generalized version of the Montesinos trick, by finding a rational tangle replacement to the unknot. An unknotting crossing can be thought of as a rational tangle replacement of a full twist, which is why we know that the surgery on $J$ must be half-integral. We found an unknotting rational tangle replacement that is three half-twists, shown in Figure \ref{fig:12n586tanglereplacement}. This means that there is a knot $J$ such that, up to mirroring $12n_{586}$, $\Sigma(12n_{586})=S^3_{\frac{d}{3}}(J)$. We found this knot $J$ by lifting the arc in Figure \ref{fig:12n586tanglereplacement} (in the same way as in Figure \ref{fig.11a255montesinos}) and the result is shown in Figure \ref{fig:12n586tanglereplacement}.

            We uploaded the diagram on the right in Figure \ref{fig:12n586tanglereplacement} to KnotFolio and found that the Alexander polynomial of $J$ is $\Delta_J(t)=t^{-11}-t^{-10}+t^{-7}-t^{-6}+t^{-4}-t^{-3}+t^{-2}-t^{-1}+1-t+t^2-t^3+t^4-t^6+t^7-t^{10}+t^{11}$. We recovered the $V_i$'s in the same manner as in Section \ref{sec:niwufla}, using \ref{ti}. Then we calculated the $d$-invariants of $L(101,3)$. These were unique up to conjugation, so using the Ni-Wu formula \ref{niwueq},
            $$
            d(\Sigma(12n_{586}),\mathfrak{t_i})=d(L(101,3),i)-2V_{\lfloor\frac{i}{3}\rfloor}
            $$
            we were able to determine the $d$-invariants of $\Sigma(12n_{586})$. Now, if $12n_{586}$ is $2$-adjacent, it has unknotting number equal to one, and so it must be half-integral surgery on some other knot $J'$, i.e. $\Sigma(12n_{586})=S^3_{\frac{101}{2}}(J')$, and another application of the Ni-Wu formula gives us
            $$
            d(\Sigma(12n_{586}),\mathfrak{t_j})=d( S^3_{\frac{101}{2}}(J'),j)=d(L(101,2),j)-2V_{\lfloor \frac{j}{2} \rfloor}.
            $$
            Rearranging, we can calculate the $V_j$'s for this new surgery description of $\Sigma(12n_{586})$
            $$
            V_{\lfloor \frac{j}{2} \rfloor}=\frac{1}{2}(d(L(101,2),j)-d(\Sigma(12n_{586}),\mathfrak{t_j}) ).
            $$

\noindent The $d$-invariants of $L(101,2)$ are:

            $\scriptstyle \left( \frac{1250}{101}, \frac{1250}{101}, \frac{1150}{101}, \frac{1152}{101}, \frac{1054}{101}, \frac{1058}{101}, \frac{962}{101}, \frac{968}{101}, \frac{874}{101}, \frac{882}{101}, \frac{790}{101}, \frac{800}{101}, \frac{710}{101}, \frac{722}{101}, \frac{634}{101}, \frac{648}{101}, \frac{562}{101}, \frac{578}{101}, \frac{494}{101}, \frac{512}{101}, \frac{430}{101}, \frac{450}{101}, \frac{370}{101}, \frac{392}{101}, \frac{314}{101}, \frac{338}{101}, \right. \\ 
            \left. \quad \frac{262}{101}, \frac{288}{101}, \frac{214}{101}, \frac{242}{101}, \frac{170}{101}, \frac{200}{101}, \frac{130}{101}, \frac{162}{101}, \frac{94}{101}, \frac{128}{101}, \frac{62}{101}, \frac{98}{101}, \frac{34}{101}, \frac{72}{101}, \frac{10}{101}, \frac{50}{101}, -\frac{10}{101}, \frac{32}{101}, -\frac{26}{101}, \frac{18}{101}, -\frac{38}{101}, \frac{8}{101}, -\frac{46}{101}, \frac{2}{101}, \right. \\ 
            \left. \quad -\frac{50}{101}, 0\right)$.

            \noindent Therefore the $V_j$'s for $S^3_{\frac{101}{2}}(J')$ are:

            $$
            \left ( 4,4,4,3,3,2,2,2,2,2,1,1,1,1,1,1,0,\dots \right )
            $$

            \noindent Using 

            \begin{equation}\label{ti2}
                V_j = \sum_{k\geq 1} ka_{j+k} ,
            \end{equation}
        
            \noindent where $a_j$ is the coefficient of $\Delta_{J'}(t)$ when it is expressed in its symmetrized form
            $$
            \Delta_{J'}(t) = a_0+\sum_{j=1}^g a_j(t^j+t^{-j}),
            $$ 
            and $g$ is the genus of $J'$, we found that 
            $$
            \Delta_{J'}(t) = t^{-15} - t^{-14} + t^{-10} - t^{-9} + t^{-5} - t^{-4} + t^{-3} - t^{-2} + 1 - t^2 + 
            t^3 - t^4 + t^5 - t^9 + t^{10} - t^{14} + t^{15}
            $$

            Neither of these evaluated at the $(2\omega)^\text{th}$ roots of unity were equal to $1$. Thus by Theorem \ref{thm-newobst}, $12n_{586}$ cannot be $2$-adjacent.

\section{Cataloging the \texorpdfstring{$2$}{2}-adjacent knots}\label{section:catalogue}

We have confirmed or obstructed $2$-adjacency in all knots with $12$ or fewer crossings. In this section, we will provide the key ingredients to prove Theorem \ref{thm2adjknots}, restated here:

\begingroup
\begin{repeattheorem}
 The following knots are $2$-adjacent: $3_1$, $4_1$, $8_{17}$, $8_{21}$, $9_{44}$, $10_{88}$, $10_{136}$, $10_{156}$, $11a_{289}$, $11n_{84}$, $11n_{125}$, $12a_{1008}$, $12a_{1249}$, $12n_{275}$, $12n_{392}$, $12n_{464}$, $12n_{482}$, $12n_{483}$, $12n_{650}$, $12n_{831}$. No other knots with $12$ or less crossings are $2$-adjacent.
\end{repeattheorem}
\endgroup

For each knot in Theorem \ref{thm2adjknots}, a diagram displaying the $2$-adjacency crossings is shown in Figure \ref{fig:2adj-knots} in the appendix.

In order to obstruct $2$-adjacency for the remaining knots, we will now prove a corollary of McCoy's (Theorem 4 in\cite{mccoy2014alternating}) that is also applicable to $2$-adjacent knots. Then we state some additional propositions due to Tao \cite{Tao1} to further obstruct $2$-adjacency.

We prove a corollary to a theorem of McCoy's (Theorem 4 in\cite{mccoy2014alternating}) that is also applicable to $2$-adjacent knots:

\begin{cor}\label{cor-mccoy}
    If $K$ is an alternating knot with positive and negative unknotting number $1$, then any minimal diagram of $K$ contains both a positive and negative unknotting crossing.
\end{cor}

\begin{proof}
    Assume that $K$ is an alternating knot with positive and negative unknotting number $1$. Then, $K$ will have signature $0$ by Proposition \ref{prop-basic-obstructions} (2). Let $D$ be an arbitrary reduced diagram for $K$. For this proof, we will use Theorem 4 from \cite{mccoy2014alternating}, specifically items \textit{i)} and \textit{iv)}. 
    Since $\sigma(K)=0$ and $K$ has negative unknotting number $1$, Theorem 4 tells us that $D$ displays a negative unknotting crossing. 

    Now, mirror $D$ to get $-D$. We now know that $-D$ has a positive unknotting crossing, since $D$ is known to have a negative one. Additionally, mirroring a knot negates its signature, so $\sigma(-D)=0$. Recall that $-D$ is a projection of $-K$, the mirror of $K$. By assumption, $K$ has positive unknotting number $1$ and negative unknotting number $1$, so $-K$ does as well.

    Knowing this, we may repeat the above logic on $-D$, a diagram of $-K$. Since $\sigma(-K)=0$ and $-K$ has negative unknotting number $1$, any minimal diagram of $-K$ must display a negative unknotting crossing, including $-D$. We have now shown that $-D$ has both a positive and negative unknotting crossing. Mirror back to obtain $D$ and we see that it also has an unknotting crossing of both signs.
\end{proof}

Tao describes restrictions on the HOMFLY-PT polynomial of $2$-adjacent knots in \cite{Tao1}.

\begin{prop} [\cite{Tao1}] \label{tao5.4}
    If $K$ is $2$-adjacent and the second coefficient of the Conway polynomial of $K$, $a_2$ is $\pm1$, then either $p_{0_K}(\ell) = \ell^{-4}+2\ell^{-2}$, $p_{0_K}(\ell) = \ell^{4}+2\ell^{2}$, or $p_{0_K}(\ell)=\ell^{-2}+1+\ell^2$.
\end{prop}

\begin{prop} [\cite{Tao1}] \label{tao5.2}
    If $K$ is $2$-adjacent and $a_2=0$, then $p_{2_K}'(i)=\pm2i\sqrt{a_4}$.
\end{prop}

\begin{prop} [\cite{Tao1}] \label{taoconway}
    Suppose $K$ is $2$-adjacent, $a_2=0$, and $a_4>0$. Then $K$ must have unknotting crossings of both signs and, by our Proposition \ref{prop-basic-obstructions} (2), we have $\sigma(K)=0$.
\end{prop}

We now have the tools necessary to prove our main result (Theorem \ref{thm2adjknots}):

\begin{proof}
    Figures showing the $2$-adjacency set for all knots listed in Theorem \ref{thm2adjknots} are included in the appendix. Proposition \ref{prop-basic-obstructions} and the determinant condition from Theorem \ref{thm-newobst} obstruct $2$-adjacency for all $\leq 12$ crossing knots except for the $2$-adjacent knots in Theorem \ref{thm2adjknots} and the following:

\begin{itemize}\label{basic-obst-knots}
    \item[]  $\{10_{82}, 10_{119}, 11a_{88}, 11a_{139}, 11a_{160}, 11a_{255}, 11n_{34}, 11n_{42}, 11n_{53}, 11n_{56}, 11n_{161}, 11n_{176},\\ 
    12a_{214}, 12a_{217}, 12a_{280}, 12a_{358}, 12a_{588}, 12a_{1228}, 12n_{45}, 12n_{176}, 12n_{258}, 12n_{265}, 12n_{306},\\ 
    12n_{313}, 12n_{370}, 12n_{430}, 12n_{431}, 12n_{434}, 12n_{449}, 12n_{566}, 12n_{586}, 12n_{610}, 12n_{616}, 12n_{620},\\
    12n_{656}, 12n_{777}.\}$
\end{itemize}

For the invariants of knots, we use the KnotInfo database. \cite{knotinfo}

From the list above, Corollary \ref{cor-mccoy} eliminates $10_{119}$, $11a_{88}$, $11a_{160}$, $12a_{214}$, $12a_{217}$, and $12a_{1228}$.

Proposition \ref{tao5.4} eliminates $11a_{139}$, $11n_{53}$, $11n_{56}$, $12a_{280}$, $12n_{45}$, $12n_{306}$, $12n_{370}$, $12n_{431}$, and $12n_{449}$.

Proposition \ref{tao5.2} eliminates $10_{82}$, $11n_{34}$, $11n_{42}$, $11n_{176}$, $12a_{588}$, $12n_{176}$, $12n_{258}$, $12n_{313}$, $12n_{430}$, $12n_{434}$, $12n_{566}$, $12n_{610}$, $12n_{616}$, and $12n_{777}$.

Proposition \ref{taoconway} eliminates $11n_{161}$ and $12n_{265}$.

Finally, by Section \ref{rulingoutalexpolys} above, we know that $11a_{255}$, $12a_{358}$, $12n_{586}$, $12n_{620}$, and $12n_{656}$ are not $2$-adjacent.
\end{proof}

\section{Seifert Surface/Conway Polynomials}\label{sectionseifert}

In this section, we discuss in more detail how we constructed the $2$-adjacent diagrams for the knots in Theorem \ref{thm2adjknots}. This was adapted from the construction used by Askitas and Kalfagianni in \cite{AK}. An appendix is included with an explicit construction of all $2$-adjacent knots up through $12$ crossings.

\begin{figure}
\centering
\hfill
\begin{minipage}[b]{0.3\textwidth}
		\centering
		\resizebox{\linewidth}{!}{
		\begin{tikzpicture}
   		 \begin{knot}[
            	clip width=15,
            	clip radius=8pt,
            	consider self intersections=no splits, 
            	end tolerance=0.1pt
            	]
        		\strand [thick] (0.5,0)
        		to [out=up,in=down] (0.75,1);
        		\strand [thick] (3.5,0)
        		to [out=up, in=down] (3.25,1);
        		\strand [thick] (0.5,5.5)
        		to [out=down, in=up] (0.75,4.5);
        		\strand [thick] (3.5,5.5)
        		to [out=down, in=up] (3.25,4.5);
    		\end{knot}
    		\begin{knot}[
            	draft mode=strands,
            	clip width=15,
            	clip radius=8pt,
            	consider self intersections=no splits, 
            	end tolerance=0.1pt
            	]
        		\strand [thick,only when rendering/.style={dashed}] (-0.5,0)
        		to [out=right, in=left] (4.5,0)
        		to [out=right, in=right] (4.5,5.5)
        		to [out=left, in=right] (-0.5,5.5)
        		to [out=left, in=left] (-0.5,0);
    		\end{knot}
    		\node[font=\Huge] at (2,2.75) {T};
		\end{tikzpicture}
		}
	\end{minipage}
	\begin{minipage}[c]{0.1\textwidth}
		\centering
		$\longrightarrow$
		\vspace{2cm}
	\end{minipage}
	\begin{minipage}[b]{0.3\textwidth}
		\centering
		\resizebox{\linewidth}{!}{
		\begin{tikzpicture}
    		\begin{knot}[
            	clip width=15,
            	clip radius=8pt,
            	consider self intersections=no splits, 
            	end tolerance=0.1pt
            	]
        		\strand [thick] (-0.5,0)
		to [out=right, in=left] (0,0);
		\strand [thick] (1,0)
		to [out=right, in=left] (3,0);
		\strand [thick] (4,0)
		to [out=right, in=left] (4.5,0)
        		to [out=right, in=right] (4.5,5.5)
        		to [out=left, in=right] (-0.5,5.5)
        		to [out=left, in=left] (-0.5,0);
        		\strand [thick] (0,0)
        		to [out=up, in=down] (0.25,1);
        		\strand [thick] (3,0)
        		to [out=up, in=down] (2.75,1);
        		\strand [thick] (1,0)
        		to [out=up, in=down] (1.25,1);
        		\strand [thick] (4,0)
	        	to [out=up, in=down] (3.75,1);
	        	\strand [thick] (0.25,4.4)
	       	to [out=up, in=down] (0,5.4)
	        	to [out=up, in=left] (0.5,6)
	        	to [out=right, in=up] (1,5.4)
	        	to [out=down, in=up] (1.25,4.4);
	        	\strand [thick] (3.75,4.4)
	        	to [out=up, in=down] (4,5.4)
	        	to [out=up, in=right] (3.5,6)
	        	to [out=left, in=up] (3,5.4)
	        	to [out=down, in=up] (2.75,4.4);
	        	\flipcrossings{2,3}
    \end{knot}
    \node[font=\Huge] at (2,2.75) {$\text{T}$};
\end{tikzpicture}
		}
	\end{minipage}
	\hfill
\vspace{10pt}
	\hfill
	\begin{minipage}[b]{0.3\textwidth}
		\centering
		\resizebox{\linewidth}{!}{
		\begin{tikzpicture} 
    		\begin{knot}[
            	clip width=15,
            	clip radius=8pt,
            	consider self intersections=no splits, 
            	end tolerance=0.1pt
            	]
        		\strand [thick] (0.5,0)
        		to [out=up,in=down] (0.5,5.5);
        		\strand [thick] (3.5,0)
        		to [out=up, in=right] (2,3)
        		to [out=left, in=up] (-1,2.25)
        		to [out=down, in=left] (2,1.5)
        		to [out=right, in=down] (3.5,4.5)
        		to [out=up, in=down] (3.5,5.5);
        		\flipcrossings{1,3}
    		\end{knot}
    		\begin{knot}[
            	draft mode=strands,
            	clip width=10,
            	clip radius=8pt,
            	consider self intersections=no splits, 
            	end tolerance=0.1pt
            	]
        		\strand [thick,only when rendering/.style={dashed}] (-0.5,0)
        		to [out=right, in=left] (4.5,0)
        		to [out=right, in=right] (4.5,5.5)
        		to [out=left, in=right] (-0.5,5.5)
        		to [out=left, in=left] (-0.5,0);
    		\end{knot}
		\end{tikzpicture}
		}
	\end{minipage}
	\begin{minipage}[c]{0.1\textwidth}
		\centering
		$\longrightarrow$
		\vspace{2cm}
	\end{minipage}
	\begin{minipage}[b]{0.3\textwidth}
		\centering
		\resizebox{\linewidth}{!}{
		\begin{tikzpicture}
   		 \begin{knot}[
 	           clip width=15,
 	           clip radius=8pt,
	            consider self intersections=no splits, 
	            end tolerance=0.1pt
 	           ]
	        	\strand [thick] (0,0)
	        	to [out=up,in=down] (0,6)
	        	to [out=up, in=left] (0.5,6.5)
       	 	to [out=right, in=up] (1,6)
        		to [out=down, in=up] (1,0)
        		to [out=right, in=left] (3,0)
        		to [out=up, in=right] (2,2.5)
        		to [out=left, in=up, looseness=0.4] (-0.5,2.25)
        		to [out=down, in=left, looseness=0.4] (2,2)
        		to [out=right, in=down] (3,4)
        		to [out=up, in=down, looseness=0.5] (4,4.5)
        		to [out=up, in=down, looseness=0.5] (3,5)
        		to [out=up, in=down, looseness=1] (3,6)
        		to [out=up, in=left] (3.5,6.5)
        		to [out=right, in=up] (4,6)
       	 	to [out=down, in=up, looseness=1] (4,5)
        		to [out=down, in=up, looseness=0.5] (3,4.5)
        		to [out=down, in=up, looseness=0.5] (4,4)
        		to [out=down, in=right] (2,1)
        		to [out=left, in=left, looseness=5] (2,3.5)
        		to [out=right, in=up] (4,0)
        		to [out=right, in=left] (4.5,0)
        		to [out=right, in=right] (4.5,5.5)
        		to [out=left, in=right] (-0.5,5.5)
        		to [out=left, in=left] (-0.5,0)
        		to [out=right, in=left] (0,0);
        		\flipcrossings{11, 12, 1, 6, 4, 9, 10, 14, 16}
    		\end{knot}
		\end{tikzpicture}
		}
	\end{minipage}
	\hfill
    \captionof{figure}{An arbitrary 2-string tangle and the knot $K_T$ resulting from the finger-move construction. (Bottom) The rational $(-2,1)$ tangle. After finger moves, it forms the $2$-adjacent knot $12n_{650}$ upon the choice of a $(+,-)$ clasp structure. A $(+,+)$ clasp structure would yield the $2$-adjacent knot $12n_{464}$, and $(-,-)$ would yield the $12n_{483}$ knot. Note that since the right arc of the tangle has writhe, twisting must be added to ensure the right band is zero-framed.}
    \label{fig:finger-move}
\end{figure}
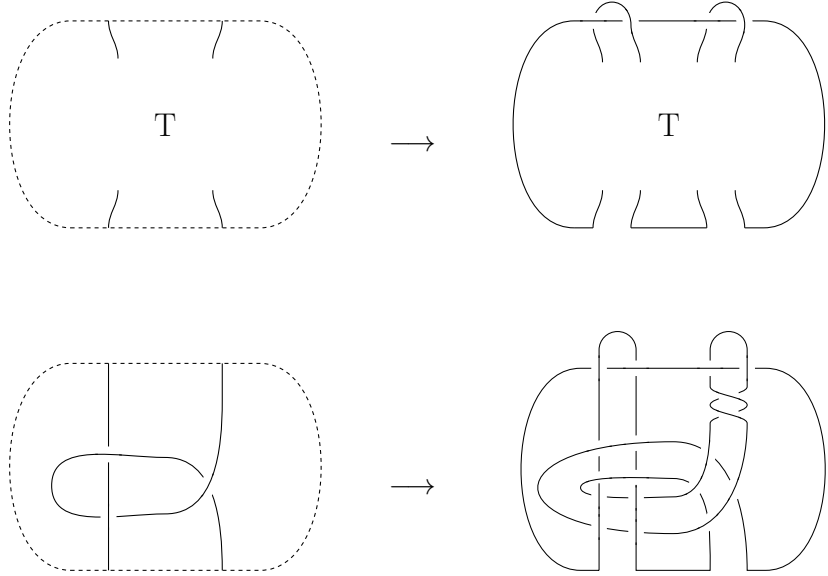

In \cite{AK}, Askitas and Kalfagianni use a construction involving finger-moves on a diagram of the unknot to characterize $n$-adjacent knots when $n\geq3$. The finger-moves are described for a certain family of trivalent spatial graphs. In particular (in \cite{AK}, Theorem 4.4), they showed that when $n\geq3$, a knot $K$ is $n$-adjacent if and only if it arises from performing these finger-moves to a \emph{Brunnian Suzuki $n$-graph} (see Section 3 of \cite{AK} for definitions). By noticing that a Seifert surface is easy to identify in such graphs, they find that the Alexander/Conway polynomial of any $n$-adjacent knot  is trivial when $n\geq 3$. This result does not apply to $2$-adjacent knots.

For $2$-adjacent knots, we look at a more narrow construction to allow us to analyze the Seifert surface. Consider a diagram of a $2$-string tangle $T$ such that neither of the two arcs is locally knotted, together with the planar circle that is the boundary of the disk containing the tangle. Perform a zero-framed finger-move along each arc, terminating in a positive or negative clasp. This yields a $2$-adjacent knot $K_{T}$, with the two clasps yielding the $2$-adjacency set. See Figure \ref{fig:finger-move}.

\begin{mydef} We say that $T$ is \emph{interleaved} if the endpoints of the two arcs alternate around the planar circle, and $T$ is non-interleaved otherwise. 
\end{mydef}

\begin{prop}\label{prop:conwaypolynomial}
The $2$-adjacent knot $K_T$ has Conway polynomial $\nabla(z) = 1 - h_1h_2\ell^2 z^4$ when $T$ is interleaved, and $\nabla(z) = 1 +h_1h_2z^2-h_1h_2 (\ell^2 -\ell) z^4$ when $T$ is non-interleaved. Here, $\ell\in\mathbb{Z}$ and $h_1,h_2\in\{+1,-1\}$.
\end{prop}

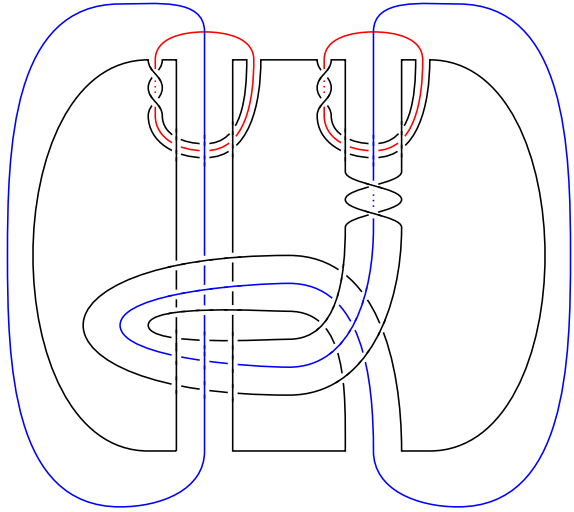
\begin{figure}
\centering
\begin{minipage}[b]{0.5\textwidth}
		\centering
		\resizebox{\linewidth}{!}{
		\begin{tikzpicture}
   		 \begin{knot}[
 	           clip width=20,
 	           clip radius=2.5pt,
	            consider self intersections=no splits, 
	            end tolerance=0.1pt
 	           ]
	        	\strand [thick] (0,0)
	        	      to [out=up,in=down] (0,7)
	        	      to [out=left, in=right] (-0.25,7)
       	 	        to [out=down, in=up] (-0.5,6.5)
                    to [out=down, in=up] (-0.25,6)
                    to [out=down, in=left] (0.48,5.5)
                    to [out=right, in=left] (0.52,5.5)
                    to [out=right, in=down] (1.25,7)
                    to [out=left, in=right] (1,7)
        		to [out=down, in=up] (1,0)
        		to [out=right, in=left] (3,0)
        		to [out=up, in=right] (2,2.5)
        		to [out=left, in=up, looseness=0.4] (-0.5,2.25)
        		to [out=down, in=left, looseness=0.4] (2,2)
        		to [out=right, in=down] (3,4)
        		to [out=up, in=down, looseness=0.5] (4,4.5)
        		to [out=up, in=down, looseness=0.5] (3,5)
        		      to [out=up,in=down] (3,7)
	        	      to [out=left, in=right] (2.75,7)
       	 	        to [out=down, in=up] (2.5,6.5)
                    to [out=down, in=up] (2.75,6)
                    to [out=down, in=left] (3.48,5.5)
                    to [out=right, in=left] (3.52,5.5)
                    to [out=right, in=down] (4.25,7)
                    to [out=left, in=right] (4,7)
       	 	  to [out=down, in=up, looseness=1] (4,5)
        		to [out=down, in=up, looseness=0.5] (3,4.5)
        		to [out=down, in=up, looseness=0.5] (4,4)
        		to [out=down, in=right] (2,1)
        		to [out=left, in=left, looseness=5] (2,3.5)
        		to [out=right, in=up] (4,0)
        		to [out=right, in=left] (4.5,0)
        		      to [out=right, in=right] (4.5,7)
        		      to [out=down, in=right] (3.52,5.25)
                    to [out=left, in=right] (3.48,5.25)
                    to [out=left, in=down] (2.5,6)
                    to [out=up, in=down] (2.75,6.5)
                    to [out=up, in=down] (2.5,7)
                    to [out=left, in=right] (1.5,7)
                    to [out=down, in=right] (0.52,5.25)
                    to [out=left, in=right] (0.48,5.25)
                    to [out=left, in=down] (-0.5,6)
                    to [out=up, in=down] (-0.25,6.5)
                    to [out=up, in=down] (-0.5,7)
        		to [out=left, in=left] (-0.5,0)
        		to [out=right, in=left] (0,0);
                \strand[red,thick,only when rendering/.style={dotted}] (-0.375,6.73)
                to [out=down, in=up] (-0.375,6.27);
                \strand[red,thick] (-0.375,7)
                to [out=down, in=up] (-0.375,6.77);
                \strand[red,thick] (-0.375,6.22)
                to [out=down, in=up] (-0.375,6)
                to [out=down, in=left] (0.48,5.375)
                to [out=right, in=left] (0.52,5.375)
                to [out=right, in=down] (1.375,7)
                to [out=up, in=right] (0.5,7.5)
                to [out=left, in=up] (-0.375,7);
                \strand[red,thick,only when rendering/.style={dotted}] (2.625,6.73)
                to [out=down, in=up] (2.625,6.27);
                \strand[red,thick] (2.625,7)
                to [out=down, in=up] (2.625,6.77);
                \strand[red,thick] (2.625,6.22)
                to [out=down, in=up] (2.625,6)
                to [out=down, in=left] (3.48,5.375)
                to [out=right, in=left] (3.52,5.375)
                to [out=right, in=down] (4.375,7)
                to [out=up, in=right] (3.5,7.5)
                to [out=left, in=up] (2.625,7);
                \strand[blue,thick] (0.5,0)
                to [out=up, in=down] (0.5,7)
                to [out=up, in=down] (0.5,7.5)
                to [out=up, in=right] (-1,8)
                to [out=left, in=up] (-3,3.5)
                to [out=down, in=left] (-1,-1)
                to [out=right, in=down] (0.5,0);
                \strand[blue,thick] (3.5,0)
                to [out=up, in=right] (2,3)
                to [out=left, in=up, looseness=0.5] (-1,2.25)
                to [out=down, in=left, looseness=0.5] (2,1.5)
                to [out=right, in=down] (3.5,4)
                to [out=up, in=down] (3.5,4.23);
                \strand[blue,thick,only when rendering/.style={dotted}] (3.5,4.27)
                to [out=up, in=down] (3.5,4.73);
                \strand[blue,thick] (3.5,4.77)
                to [out=up, in=down] (3.5,7)
                to [out=up, in=down] (3.5,7.5)
                to [out=up, in=left] (5,8)
                to [out=right, in=up] (7,3.5)
                to [out=down, in=right] (5,-1)
                to [out=left, in=down] (3.5,0);
        		\flipcrossings{2,5,8,10,12,13,14,17,20,22,23,24,26,29,35,36,37,42,44,45,46,47,48,49,51}
    		\end{knot}
		\end{tikzpicture}
		}
	\end{minipage}
    \captionof{figure}{The knot $12n_{650}$ with a Seifert surface and the curves $\alpha_1, \alpha_2$ (red) and $\beta_1,\beta_2$ (blue) shown.}
    \label{fig:seifertsurfacewithcurves}
\end{figure}

\begin{proof}
Recall that the Seifert matrix of a Seifert surface $F, \partial F=K$, is defined by $V_{ij}:=\operatorname{lk}( \gamma_i, \gamma_j^\#)$, where $\gamma_i, \gamma_j$ are generators of $H_1(F)$ and $\gamma_j^\#$ is the push-off of $\gamma_j$ in the normal direction. The Alexander polynomial of the knot $K$ is given by $\det(V-tV^T)$. 

The diagram for $K_T$ resulting from its Suzuki graph yields a disk-band configuration of a genus two Seifert surface $F$, as in Figure \ref{fig:seifertsurfacewithcurves}. Label the generators of $H_1(F)$ as $\alpha_1, \alpha_2, \beta_1, \beta_2$. Pick orientations of $\beta_1$ and $\beta_2$, then fix orientations of $\alpha_1$ and $\alpha_2$ such that $\operatorname{lk}(\beta_1,\alpha_1^\#)=1$ and $\operatorname{lk}(\beta_2,\alpha_2^\#)=1$. When $G$ is non-interleaved, the Seifert matrix $V$ is given by
\[
    V=
    \begin{blockarray}{ccccc}
    \alpha_1^\# & \alpha_2^\# & \beta_1^\# & \beta_2^\# \\
    \begin{block}{[cccc]c}
    h_1 & 0 & 0 & 0 & \alpha_1\\
    0 & h_2 & 0 & 0 & \alpha_2\\
    1 & 0 & 0 & \ell & \beta_1\\
    0 & 1 & \ell & 0 & \beta_2\\
    \end{block}
    \end{blockarray}
    \]
Here, $h_1$ and $h_2$ are the signs of the clasps, and $\ell$  is the linking number of the two arcs.

When $T$ is interleaved, the Seifert matrix $V$ is given by
\[
    V=
    \begin{blockarray}{ccccc}
    \alpha_1^\# & \alpha_2^\# & \beta_1^\# & \beta_2^\# \\
    \begin{block}{[cccc]c}
    h_1 & 0 & 0 & 0 & \alpha_1\\
    0 & h_2 & 0 & 0 & \alpha_2\\
    1 & 0 & 0 & \ell & \beta_1\\
    0 & 1 & \ell-1 & 0 & \beta_2\\
    \end{block}
    \end{blockarray}
    \]
Most values in these matrices are apparent from Figure \ref{fig:seifertsurfacewithcurves}, but some explanation for the presence of $\ell-1$ may be needed. In the interleaved case, the arcs $\beta_1$ and $\beta_2$ cross an odd number of times on the bands, and meet at a point in the Seifert surface. Make a choice of orientation on $\beta_1$ and $\beta_2$. Then, when each is pushed off, the point where they meet in the Seifert surface contributes either a positive or a negative crossing to the overall linking number such that $\operatorname{lk}(\beta_1,\beta_2^\#)$ and $\operatorname{lk}(\beta_2,\beta_1^\#)$ differ by $1$. Without loss of generality, we label $\beta_1$ and $\beta_2$ (and therefore $\alpha_1$ and $\alpha_2$) such that $\ell$ is the larger value. 

Taking the determinant $\det(V-tV^T)$, and applying the normalization $z=(t^{1/2} - t^{-1/2})$, as well as dividing by $t^2$, we obtain the polynomials
\[
	\nabla(z) = 1 - h_1h_2\ell^2 z^4 \qquad \text{and} \qquad  \nabla(z) = 1 +h_1h_2z^2-h_1h_2 (\ell^2 -\ell) z^4 
\]
when $T$ is non-interleaved, and interleaved, respectively. 
\end{proof}

This construction creates $2$-adjacent knots $K_T$ where the corresponding tangle $T$ is fully contained within the outer circle. All of these will have genus $\leq2$. There exist $2$-adjacent knots of genus $>2$, so we know there are $2$-adjacent knots where a corresponding $T$, if it exists, is not contained. Consider the example of $8_{17}$, where a non-contained finger-move diagram is shown in the Appendix Figure \ref{fig:2adj-knots}. It is not yet known if all $2$-adjacent knots can be realized as finger-move diagrams, allowing for uncontained variants.
    
\section{Conclusion}

 Now that all knots with $12$ or fewer crossing have a known $2$-adjacency status, the next logical step would be to try and categorize the knots with $13$ crossings, but it would require new methods to do so. Among the $13$ crossing knots, there are $13$ known to be $2$-adjacent and $74$ whose $2$-adjacency status is unknown (see the appendix for the full list). Due to the complexity of performing the method outlined in Section \ref{rulingoutalexpolys}, it has not been applied here to any $13$ crossing knots. There are many $13$ crossing knots of unknown $2$-adjacency status whose unknotting number is also not known, making the above method even more difficult.

As can be seen from this paper, $2$-adjacent knots are very different from other $n$-adjacent knots. When Lidman and Moore \cite{lidman2023adjacencythreemanifoldsbrunnianlinks} expanded the results on $n$-adjacency of knots by Askitas and Kalfagianni to $n$-adjacency of $3$-manifolds, many of their results applied only for $n>2$. The case of $2$-adjacent $3$-manifolds is currently not studied.

The evidence suggests that alternating $2$-adjacent knots should be easy to identify. All known alternating $2$-adjacent knots have their $2$-adjacency visible in those minimal diagrams checked. Proposition \ref{cor-mccoy} gives us that alternating $2$-adjacent knots with positive and negative unknotting number $1$ display a positive and negative unknotting crossing in any minimal diagram. There is no guarantee that these crossings are still a $2$-adjacency, but they are in every minimal diagram inspected for this paper. We conjecture that those crossings are still a $2$-adjacency. Additionally, we conjecture that the only alternating knot with a $2$-adjacency of the same crossing sign is the trefoil. No others have been found, despite a brute-force search through 16 crossings. It is difficult to prove this statement because we would need to use a quality specific to alternating knots. The knot $8_{21}$ has a $2$-adjacency set of the same sign, but it is both quasi-alternating and almost alternating, so a condition that also applies to quasi-alternating or almost alternating knots would be insufficient to prove this conjecture. As evidence in support of the conjecture, Askitas and Kalfagianni previously found that there are no non-trivial alternating $n$-adjacent knots ($n>2$). If any did exist, it would have a $2$-adjacency set of the same crossing sign as a subset of the $n$-adjacency.

Together, these two conjectures would imply a version of Kohn's conjecture for $2$-adjacent alternating knots: 

\begin{conj}
    All alternating non-trivial $2$-adjacent knots have a $2$-adjacency set in every minimal diagram.
\end{conj}

Askitas and Kalfagianni made a construction that includes all $n$-adjacent knots when $n>2$. The construction in this paper is a more general version of theirs and may include all $2$-adjacent knots, allowing for non-contained diagrams, but that is unproven. 

It seems like further research into $2$-adjacent knots will require new techniques. New ways to find $2$-adjacent knots and exclude other knots would provide a stronger list of examples to spark observations. The novel method of obstructing $2$-adjacency using Heegaard Floer $d$-invariants has the potential to be expanded upon for this topic and others. 

\subsection*{Acknowledgments}  We thank Tye Lidman and Allison Moore for their helpful discussions as well as their comments on drafts of the manuscript. EM was supported by NSF Grant DMS--2105469.  JC was supported by NSF Grant DMS--2204148. 

\newpage

\section{Appendix} \label{appendix}

\begin{figure}[h]
\begin{centering}
\begin{minipage}{0.24\textwidth}
\centering
\resizebox{\linewidth}{!}{
\begin{tikzpicture} 
\path[use as bounding box] (-3.6,0) rectangle (7.5,6.5);
    \begin{knot}[
            clip width=10,
            clip radius=8pt,
            consider self intersections=no splits,
            end tolerance=0.1pt
            ]
        \strand [thick] (0,0)
        to [out=up, in=down] (3,6)
        to [out=up,in=left] (3.5,6.5)
        to [out=right,in=up] (4,6)
        to [out=down, in=up] (1,0)
        to [out=right, in=left] (3,0)
        to [out=up, in=down] (0,6)
        to [out=up, in=left] (0.5,6.5)
        to [out=right, in=up] (1,6)
        to [out=down, in=up] (4,0)
        to [out=right, in=left] (4.5,0)
        to [out=right, in=right] (4.5,5.5)
        to [out=left, in=right] (-0.5,5.5)
        to [out=left, in=left] (-0.5,0)
        to [out=right, in=left] (0,0);
        \flipcrossings{8,6}
    \end{knot}
\end{tikzpicture}
}
\caption*{$3_1$}
\end{minipage}
\hfill
\begin{minipage}{0.24\textwidth}
\centering
\resizebox{\linewidth}{!}{
\begin{tikzpicture} 
\path[use as bounding box] (-3.6,0) rectangle (7.5,6.5);
    \begin{knot}[
            clip width=10,
            clip radius=8pt,
            consider self intersections=no splits,
            end tolerance=0.1pt
            ]
        \strand [thick] (0,0)
        to [out=up, in=down] (3,6)
        to [out=up,in=left] (3.5,6.5)
        to [out=right,in=up] (4,6)
        to [out=down, in=up] (1,0)
        to [out=right, in=left] (3,0)
        to [out=up, in=down] (0,6)
        to [out=up, in=left] (0.5,6.5)
        to [out=right, in=up] (1,6)
        to [out=down, in=up] (4,0)
        to [out=right, in=left] (4.5,0)
        to [out=right, in=right] (4.5,5.5)
        to [out=left, in=right] (-0.5,5.5)
        to [out=left, in=left] (-0.5,0)
        to [out=right, in=left] (0,0);
        \flipcrossings{8,3}
    \end{knot}
\end{tikzpicture}
}
\caption*{$4_1$}
\end{minipage}
\hfill
\begin{minipage}{0.24\textwidth}
\centering
\resizebox{\linewidth}{!}{
\begin{tikzpicture} 
\path[use as bounding box] (-3.6,0) rectangle (7.5,6.5);
    \begin{knot}[
            clip width=10,
            clip radius=8pt,
            consider self intersections=no splits,
            end tolerance=0.1pt
            ]
        \strand [thick] (0,0)
        to [out=up, in=south east, looseness=0.75] (1.9,4.2)
        to [out=north west,in=up, looseness=0.6] (-2.5,2.5)
        to [out=down,in=down, looseness=0.8] (4,3)
        to [out=up, in=down] (3,6)
        to [out=up, in=left] (3.5,6.5)
        to [out=right, in=up] (4,6)
        to [out=down, in=up] (5,3)
        to [out=down, in=down, looseness=0.9] (-3.5,2.5)
        to [out=up, in=north west, looseness=0.7] (2.6,4.6)
        to [out=south east, in=up] (1,0)
        to [out=right, in=left] (3,0)
        to [out=up, in=down] (0,6)
        to [out=up, in=left] (0.5,6.5)
        to [out=right, in=up] (1,6)
        to [out=down, in=up] (4,0)
        to [out=right, in=left] (4.6,0)
        to [out=right, in=right] (4.6,5.5)
        to [out=left, in=right] (-0.5,5.5)
        to [out=left, in=left] (-0.5,0)
        to [out=right, in=left] (0,0);
        \flipcrossings{1,2,5,6,18,19,7,20,9,10,15,16,12,24}
    \end{knot}
\end{tikzpicture}
}
\caption*{$8_{17}$}
\end{minipage}
\hfill
\begin{minipage}{0.24\textwidth}
\centering
\resizebox{\linewidth}{!}{
\begin{tikzpicture} 
\path[use as bounding box] (-3.6,0) rectangle (7.5,6.5);
    \begin{knot}[
            clip width=10,
            clip radius=8pt,
            consider self intersections=no splits,
            end tolerance=0.1pt
            ]
        \strand [thick] (0,0)
        to [out=up, in=down, looseness=0.6] (2.5,3)
        to [out=up,in=down, looseness=0.6] (0,6)
        to [out=up, in=left] (0.5,6.5)
        to [out=right, in=up] (1,6)
        to [out=south, in=up, looseness=0.6] (3.3,3)
        to [out=down, in=up, looseness=0.6] (1,0)
        to [out=right, in=left] (3,0)
        to [out=up, in=down, looseness=0.6] (0.7,3)
        to [out=up, in=down, looseness=0.6] (3,6)
        to [out=up, in=left] (3.5,6.5)
        to [out=right, in=up] (4,6)
        to [out=down, in=up, looseness=0.6] (1.5,3)
        to [out=down, in=up, looseness=0.6] (4,0)
        to [out=right, in=left] (4.3,0)
        to [out=right, in=right] (4.3,5.5)
        to [out=left, in=right] (-0.3,5.5)
        to [out=left, in=left] (-0.3,0)
        to [out=right, in=left] (0,0);
        \flipcrossings{3,4,6,7,8,12}
    \end{knot}
\end{tikzpicture}
}
\caption*{$8_{21}$}
\end{minipage}

\begin{minipage}{0.24\textwidth}
\centering
\resizebox{\linewidth}{!}{
\begin{tikzpicture} 
\path[use as bounding box] (-3.6,0) rectangle (7.5,6.5);
    \begin{knot}[
            clip width=10,
            clip radius=8pt,
            consider self intersections=no splits,
            end tolerance=0.1pt
            ]
        \strand [thick] (0,0)
        to [out=up, in=down, looseness=0.6] (2.5,3)
        to [out=up,in=down, looseness=0.6] (0,6)
        to [out=up, in=left] (0.5,6.5)
        to [out=right, in=up] (1,6)
        to [out=south, in=up, looseness=0.6] (3.3,3)
        to [out=down, in=up, looseness=0.6] (1,0)
        to [out=right, in=left] (3,0)
        to [out=up, in=down, looseness=0.6] (0.7,3)
        to [out=up, in=down, looseness=0.6] (3,6)
        to [out=up, in=left] (3.5,6.5)
        to [out=right, in=up] (4,6)
        to [out=down, in=up, looseness=0.6] (1.5,3)
        to [out=down, in=up, looseness=0.6] (4,0)
        to [out=right, in=left] (4.3,0)
        to [out=right, in=right] (4.3,5.5)
        to [out=left, in=right] (-0.3,5.5)
        to [out=left, in=left] (-0.3,0)
        to [out=right, in=left] (0,0);
        \flipcrossings{3,4,6,7,8,11}
    \end{knot}
\end{tikzpicture}
}
\caption*{$9_{44}$}
\end{minipage}
\hfill
\begin{minipage}{0.24\textwidth}
\centering
\resizebox{\linewidth}{!}{
\begin{tikzpicture} 
\path[use as bounding box] (-3.6,0) rectangle (7.5,6.5);
    \begin{knot}[
            clip width=10,
            clip radius=8pt,
            consider self intersections=no splits,
            end tolerance=0.1pt
            ]
        \strand [thick] (0,0)
        to [out=up, in=south west, looseness=0.3] (3,1.5)
        to [out=north east, in=down, looseness=0.5] (3.9,4)
        to [out=north west,in=right, looseness=0.6] (1.2,2.9)
        to [out=left,in=down, looseness=1] (0.6,3.4)
        to [out=up, in=down, looseness=0.7] (3,6)
        to [out=up, in=left] (3.5,6.5)
        to [out=right, in=up] (4,6)
        to [out=down, in=left, looseness=0.8] (1.3,3.4)
        to [out=right, in=north west, looseness=0.7] (4.4,4.4)
        to [out=south east, in=north east, looseness=0.4] (3.7,1.4)
        to [out=south west, in=up, looseness=0.3] (1,0)
        to [out=right, in=left] (3,0)
        to [out=up, in=down, looseness=0.8] (1.5,1.6)
        to [out=up, in=down, looseness=0.5] (6.5,3)
        to [out=up, in=right, looseness=0.5] (5,3.5)
        to [out=left, in=down, looseness=0.5] (0,4.25)
        to [out=up, in=down, looseness=0.5] (0.7,4.625)
        to [out=up,in=down, looseness=0.5] (0,5)
        to [out=up, in=down] (0,6)
        to [out=up, in=left] (0.35,6.5)
        to [out=right, in=up] (0.7,6)
        to [out=down, in=up] (0.7,5)
        to [out=down, in=up, looseness=0.5] (0,4.625)
        to [out=down, in=up, looseness=0.5] (0.7,4.25)
        to [out=down, in=left, looseness=0.2] (5.2,3.9)
        to [out=right, in=up] (7,3.3)
        to [out=down, in=up] (7,3)
        to [out=down, in=right](5,2)
        to [out=left, in=up, looseness=0.4] (2.2,1.4)
        to [out=down, in=up, looseness=0.4] (4,0)
        to [out=right, in=left] (4.5,0)
        to [out=right, in=right] (4.5,5.5)
        to [out=left, in=right] (-0.5,5.5)
        to [out=left, in=left] (-0.5,0)
        to [out=right, in=left] (0,0);
        \flipcrossings{5,18,3,4,19,20,9,10,12,13,11,30,23,25,28,6,17}
    \end{knot}
\end{tikzpicture}
}
\caption*{$10_{88}$}
\end{minipage}
\hfill
\begin{minipage}{0.24\textwidth}
\centering
\resizebox{\linewidth}{!}{
\begin{tikzpicture} 
\path[use as bounding box] (-3.6,0) rectangle (7.5,6.5);
    \begin{knot}[
            clip width=10,
            clip radius=8pt,
            consider self intersections=no splits,
	 end tolerance=0.1pt
            ]
        \strand [thick] (0,0)
        to [out=up, in=down, looseness=0.6] (2.5,3)
        to [out=up,in=down, looseness=0.6] (0,6)
        to [out=up, in=left] (0.5,6.5)
        to [out=right, in=up] (1,6)
        to [out=south, in=up, looseness=0.6] (3.3,3)
        to [out=down, in=up, looseness=0.6] (1,0)
        to [out=right, in=left] (3,0)
        to [out=up, in=down, looseness=0.6] (0.7,3)
        to [out=up, in=down, looseness=0.6] (3,6)
        to [out=up, in=left] (3.5,6.5)
        to [out=right, in=up] (4,6)
        to [out=down, in=up, looseness=0.6] (1.5,3)
        to [out=down, in=up, looseness=0.6] (4,0)
        to [out=right, in=left] (4.3,0)
        to [out=right, in=right] (4.3,5.5)
        to [out=left, in=right] (-0.3,5.5)
        to [out=left, in=left] (-0.3,0)
        to [out=right, in=left] (0,0);
        \flipcrossings{3,4,6,7,5,11}
    \end{knot}
\end{tikzpicture}
}
\caption*{$10_{136}$}
\end{minipage}
\hfill
\begin{minipage}{0.24\textwidth}
\centering
\resizebox{\linewidth}{!}{
\begin{tikzpicture} 
\path[use as bounding box] (-3.6,0) rectangle (7.5,6.5);
    \begin{knot}[
            clip width=10,
            clip radius=8pt,
            consider self intersections=no splits,
	 end tolerance=0.1pt
            ]
        \strand [thick] (0,0)
        to [out=up, in=south east, looseness=0.75] (1.9,4.2)
        to [out=north west,in=up, looseness=0.6] (-2.5,2.5)
        to [out=down,in=down, looseness=0.8] (4,3)
        to [out=up, in=down] (3,6)
        to [out=up, in=left] (3.5,6.5)
        to [out=right, in=up] (4,6)
        to [out=down, in=up] (5,3)
        to [out=down, in=down, looseness=0.9] (-3.5,2.5)
        to [out=up, in=north west, looseness=0.7] (2.6,4.6)
        to [out=south east, in=up] (1,0)
        to [out=right, in=left] (3,0)
        to [out=up, in=down] (0,6)
        to [out=up, in=left] (0.5,6.5)
        to [out=right, in=up] (1,6)
        to [out=down, in=up] (4,0)
        to [out=right, in=left] (4.6,0)
        to [out=right, in=right] (4.6,5.5)
        to [out=left, in=right] (-0.5,5.5)
        to [out=left, in=left] (-0.5,0)
        to [out=right, in=left] (0,0);
        \flipcrossings{1,2,5,6,18,19,7,20,9,10,15,16,12,23}
    \end{knot}
\end{tikzpicture}
}
\caption*{$10_{156}$}
\end{minipage}

\begin{minipage}{0.24\textwidth}
\centering
\resizebox{\linewidth}{!}{
\begin{tikzpicture} 
\path[use as bounding box] (-3.6,0) rectangle (7.5,6.5);
    \begin{knot}[
            clip width=10,
            clip radius=8pt,
            consider self intersections=no splits,
	 end tolerance=0.1pt
            ]
        \strand [thick] (0,0)
        to [out=up, in=left, looseness=0.7] (1.3,2)
        to [out=right, in=left, looseness=0.6] (4.2,2.2)
        to [out=right,in=down, looseness=0.3] (6.5,2.5)
        to [out=up, in=right, looseness=0.3] (2,3)
        to [out=left, in=down, looseness=0.4] (-3.5,4)
        to [out=up, in=left, looseness=0.5] (0.56,4.46)
        to [out=right, in=down, looseness=0.8] (0.6,4.5)
        to [out=up, in=down] (0.6,6)
        to [out=up, in=left] (0.85,6.5)
        to [out=right, in=up] (1.1,6)
        to [out=down, in=up, looseness=0.6] (1.1,4.2)
        to [out=down, in=up, looseness=0.4] (-2.5,4)
        to [out=down, in=left, looseness=0.3] (2,3.5)
        to [out=right, in=up, looseness=0.5] (7.5,2.5)
        to [out=down, in=north east, looseness=0.6] (1.1,1.1)
        to [out=south west, in=up] (0.7,0)
        to [out=right, in=left] (3.2,0)
        to [out=up, in=right] (2.6,2.35)
        to [out=left, in=right, looseness=0.5] (1.7,0.3)
        to [out=left, in=south east, looseness=1] (-0.7,1.5)
        to [out=north west, in=south west, looseness=0.7] (-0.7,4.9)
        to [out=north east, in=left] (2,5)
        to [out=right, in=down] (3,6)
        to [out=up, in=left] (3.25,6.5)
        to [out=right, in=up] (3.5,6)
        to [out=down, in=right] (2,4.5)
        to [out=left, in=north east] (-0.2,4.8)
        to [out=south west, in=left] (1.4,0.7)
        to [out=right, in=left, looseness=0.8] (2.5,2.75)
        to [out=right, in=up] (4,0)
        to [out=right, in=left] (4.5,0)
        to [out=right, in=right] (4.5,5.5)
        to [out=left, in=right] (-0.5,5.5)
        to [out=left, in=left] (-0.5,0)
        to [out=right, in=left] (0,0);
        \flipcrossings{3,6,1,2,32,33,9,10,24,17,21,12,13,11,25,31,7,36,34,28,26,22}
    \end{knot}
\end{tikzpicture}
}
\caption*{$11a_{289}$}
\end{minipage}
\hfill
\begin{minipage}{0.24\textwidth}
\centering
\resizebox{\linewidth}{!}{
\begin{tikzpicture} 
\path[use as bounding box] (-3.6,0) rectangle (7.5,6.5);
    \begin{knot}[
            clip width=10,
            clip radius=8pt,
            consider self intersections=no splits,
	 end tolerance=0.1pt
            ]
        \strand [thick] (0,0)
        to [out=up, in=down, looseness=0.4] (3,1.6)
        to [out=up, in=down, looseness=0.4] (0,3.2)
        to [out=up, in=down, looseness=0.4] (3,4.8)
        to [out=up, in=down] (3,6)
        to [out=up, in=left] (3.5,6.5)
        to [out=right, in=up] (4,6)
        to [out=down, in=up] (4,4.8)
        to [out=down, in=up, looseness=0.4] (1,3.2)
        to [out=down, in=up, looseness=0.4] (4,1.6)
        to [out=down, in=up, looseness=0.4] (1,0)
        to [out=right, in=left] (3,0)
        to [out=up, in=down, looseness=0.4] (0,1.6)
        to [out=up, in=down, looseness=0.4] (3,3.2)
        to [out=up, in=down, looseness=0.4] (0,4.8)
        to [out=up, in=down] (0,6)
        to [out=up, in=left] (0.5,6.5)
        to [out=right, in=up] (1,6)
        to [out=down, in=up] (1,4.8)
        to [out=down, in=up, looseness=0.4] (4,3.2)
        to [out=down, in=up, looseness=0.4] (1,1.6)
        to [out=down, in=up, looseness=0.4] (4,0)
        to [out=right, in=right] (4.5,5.5)
        to [out=left, in=right] (-0.5,5.5)
        to [out=left, in=left] (-0.5,0)
        to [out=right, in=left] (0,0);
        \flipcrossings{3,4,11,12,8,16}
    \end{knot}
\end{tikzpicture}
}
\caption*{$11n_{84}$}
\end{minipage}
\hfill
\begin{minipage}{0.24\textwidth}
\centering
\resizebox{\linewidth}{!}{
\begin{tikzpicture} 
\path[use as bounding box] (-3.6,0) rectangle (7.5,6.5);
    \begin{knot}[
            clip width=10,
            clip radius=8pt,
            consider self intersections=no splits,
	 end tolerance=0.1pt
            ]
        \strand [thick] (0,0)
        to [out=up, in=left, looseness=0.8] (2,1.2)
        to [out=right, in=down, looseness=0.4] (4.4,1.5)
        to [out=up,in=down, looseness=0.4] (4.4,4.3)
        to [out=up,in=right, looseness=0.4] (2,4.3)
        to [out=left,in=down, looseness=1] (0,6)
        to [out=up, in=left] (0.5,6.5)
        to [out=right, in=up] (1,6)
        to [out=down, in=left] (2,5)
        to [out=right, in=up, looseness=1] (5,4.3)
        to [out=down, in=up, looseness=0.4] (5,1.5)
        to [out=down, in=right, looseness=0.7] (2,0.5)
        to [out=left, in=up, looseness=0.7] (1,0)
        to [out=right, in=left] (3,0)
        to [out=up, in=left] (3.3,2.7)
        to [out=right, in=down] (7,3.5)
        to [out=up, in=right, looseness=0.5] (3.8,3.7)
        to [out=left, in=down] (4,6)
        to [out=up, in=right] (3.5,6.5)
        to [out=left, in=up] (3,6)
        to [out=down, in=left] (3.3,3.2)
        to [out=right, in=up] (7,2.5)
        to [out=down, in=right, looseness=0.5] (3.8,2.2)
        to [out=left, in=up] (4,0)
        to [out=right, in=left] (4.6,0)
        to [out=right, in=right] (4.6,5.5)
        to [out=left, in=right] (-0.5,5.5)
        to [out=left, in=left] (-0.5,0)
        to [out=right, in=left] (0,0);
        \flipcrossings{1, 2, 17, 18, 7, 8, 11, 12, 10, 23, 20, 25, 19}
    \end{knot}
\end{tikzpicture}
}
\caption*{$11n_{125}$}
\end{minipage}
\hfill
\begin{minipage}{0.24\textwidth}
\centering
\resizebox{\linewidth}{!}{
\begin{tikzpicture} 
\path[use as bounding box] (-3.6,0) rectangle (7.5,6.5);
    \begin{knot}[
            clip width=10,
            clip radius=8pt,
            consider self intersections=no splits,
	 end tolerance=0.1pt
            ]
        \strand [thick] (2,0)
        to [out=right, in=down] (6,2)
        to [out=up, in=right] (5,4)
        to [out=left, in=up] (4.5,2)
        to [out=down, in=right] (4,0.5)
        to [out=left, in=down] (3.5,2)
        to [out=up, in=down] (3.5,5.5)
        to [out=up, in=right] (3,6)
        to [out=left, in=up] (2.5,5.5)
        to [out=down, in=up] (2.5,2)
        to [out=down, in=right] (2,0.5)
        to [out=left, in=down] (0,2)
        to [out=up, in=left] (2,5)
        to [out=right, in=left] (3.7,5)
        to [out=right, in=up] (4.1,4.5)
        to [out=down, in=right] (3.7,4)
        to [out=left, in=right] (2,4)
        to [out=left, in=up] (1,2)
        to [out=down, in=left] (3,1.5)
        to [out=right, in=left] (4.7,1.5)
        to [out=right, in=down] (5,2)
        to [out=up, in=right] (4.7,2.5)
        to [out=left, in=right] (2,2.5)
        to [out=left, in=left] (2,0);
        \flipcrossings{1,3,5,7,9,11}
    \end{knot}
    \begin{knot}[
            draft mode=strands,
            clip width=10,
            clip radius=8pt,
            consider self intersections=false,
	 end tolerance=0.1pt
            ]
        \strand[thick,red](1.1,0.55)
        to [out=up, in=left] (1.4,0.85)
        to [out=right, in=up] (1.7,0.55)
        to [out=down, in=right] (1.4,0.25)
        to [out=left, in=down] (1.1,0.55);
    \end{knot}
    \begin{knot}[
            draft mode=strands,
            clip width=10,
            clip radius=8pt,
            consider self intersections=false,
	 end tolerance=0.1pt
            ]
        \strand[thick,red](2.2,2.5)
        to [out=up, in=left] (2.5,2.8)
        to [out=right, in=up] (2.8,2.5)
        to [out=down, in=right] (2.5,2.2)
        to [out=left, in=down] (2.2,2.5);
    \end{knot}
\end{tikzpicture}
}
\caption*{$12a_{1008}$}
\end{minipage}

\begin{minipage}{0.24\textwidth}
\centering
\resizebox{\linewidth}{!}{
\begin{tikzpicture} 
\path[use as bounding box] (-3.6,0) rectangle (7.5,6.5);
    \begin{knot}[
            clip width=10,
            clip radius=8pt,
            consider self intersections=no splits,
	 end tolerance=0.1pt
            ]
        \strand [thick] (3,0)
        to [out=left, in=down] (2,1)
        to [out=up, in=down] (2,3.5)
        to [out=up, in=right] (1.5,4)
        to [out=left, in=right] (-0.5,4)
        to [out=left, in=down] (-1,4.5)
        to [out=up, in=left] (-0.5,5)
        to [out=right, in=left] (4.5,5)
        to [out=right, in=up] (5,3.5)
        to [out=down, in=right] (4.5,2)
        to [out=left, in=right] (1.5,2)
        to [out=left, in=left] (1.5,5.5)
        to [out=right, in=left] (5,5.5)
        to [out=right, in=up] (6,3.5)
        to [out=down, in=right] (5,1.5)
        to [out=left, in=right] (3,1.5)
        to [out=left, in=right] (2.5,3)
        to [out=left, in=right] (0.5,3)
        to [out=left, in=left] (0.5,6)
        to [out=right, in=left] (3,6)
        to [out=right, in=up] (4,3.5)
        to [out=down, in=up] (4,1)
        to [out=down, in=right] (3,0);
        \flipcrossings{2,4,5,9,11}
    \end{knot}
    \begin{knot}[
            draft mode=strands,
            clip width=10,
            clip radius=8pt,
            consider self intersections=false,
	        end tolerance=0.1pt
            ]
        \strand[thick,red](3.7,1.5)
        to [out=up, in=left] (4,1.8)
        to [out=right, in=up] (4.3,1.5)
        to [out=down, in=right] (4,1.2)
        to [out=left, in=down] (3.7,1.5);
    \end{knot}
    \begin{knot}[
            draft mode=strands,
            clip width=10,
            clip radius=8pt,
            consider self intersections=false,
	        end tolerance=0.1pt
            ]
        \strand[thick,red](0.2,4)
        to [out=up, in=left] (0.5,4.3)
        to [out=right, in=up] (0.8,4)
        to [out=down, in=right] (0.5,3.7)
        to [out=left, in=down] (0.2,4);
    \end{knot}
\end{tikzpicture}
}
\caption*{$12a_{1249}$}
\end{minipage}
\hfill
\begin{minipage}{0.24\textwidth}
\centering
\resizebox{\linewidth}{!}{
\begin{tikzpicture} 
\path[use as bounding box] (-3.6,0) rectangle (7.5,6.5);
    \begin{knot}[
            clip width=10,
            clip radius=8pt,
            consider self intersections=no splits,
	 end tolerance=0.1pt
            ]
        \strand [thick] (0,0)
        to [out=up, in=down, looseness=0.4] (3,1.6)
        to [out=up, in=down, looseness=0.4] (0,3.2)
        to [out=up, in=down, looseness=0.4] (3,4.8)
        to [out=up, in=down] (3,6)
        to [out=up, in=left] (3.5,6.5)
        to [out=right, in=up] (4,6)
        to [out=down, in=up] (4,4.8)
        to [out=down, in=up, looseness=0.4] (1,3.2)
        to [out=down, in=up, looseness=0.4] (4,1.6)
        to [out=down, in=up, looseness=0.4] (1,0)
        to [out=right, in=left] (3,0)
        to [out=up, in=down, looseness=0.4] (0,1.6)
        to [out=up, in=down, looseness=0.4] (3,3.2)
        to [out=up, in=down, looseness=0.4] (0,4.8)
        to [out=up, in=down] (0,6)
        to [out=up, in=left] (0.5,6.5)
        to [out=right, in=up] (1,6)
        to [out=down, in=up] (1,4.8)
        to [out=down, in=up, looseness=0.4] (4,3.2)
        to [out=down, in=up, looseness=0.4] (1,1.6)
        to [out=down, in=up, looseness=0.4] (4,0)
        to [out=right, in=right] (4.5,5.5)
        to [out=left, in=right] (-0.5,5.5)
        to [out=left, in=left] (-0.5,0)
        to [out=right, in=left] (0,0);
        \flipcrossings{3,4,11,12,8,15}
    \end{knot}
\end{tikzpicture}
}
\caption*{$12n_{275}$}
\end{minipage}
\hfill
\begin{minipage}{0.24\textwidth}
\centering
\resizebox{\linewidth}{!}{
\begin{tikzpicture} 
\path[use as bounding box] (-3.6,0) rectangle (7.5,6.5);
   \begin{knot}[
            clip width=10,
            clip radius=8pt,
            consider self intersections=no splits,
            end tolerance=0.1pt
            ]
        \strand [thick] (0,0)
        to [out=up, in=south west, looseness=0.3] (3,1.5)
        to [out=north east, in=down, looseness=0.5] (3.9,4)
        to [out=north west,in=right, looseness=0.6] (1.2,2.9)
        to [out=left,in=down, looseness=1] (0.6,3.4)
        to [out=up, in=down, looseness=0.7] (3,6)
        to [out=up, in=left] (3.5,6.5)
        to [out=right, in=up] (4,6)
        to [out=down, in=left, looseness=0.8] (1.3,3.4)
        to [out=right, in=north west, looseness=0.7] (4.4,4.4)
        to [out=south east, in=north east, looseness=0.4] (3.7,1.4)
        to [out=south west, in=up, looseness=0.3] (1,0)
        to [out=right, in=left] (3,0)
        to [out=up, in=down, looseness=0.8] (1.5,1.6)
        to [out=up, in=down, looseness=0.5] (6.5,3)
        to [out=up, in=right, looseness=0.5] (5,3.5)
        to [out=left, in=down, looseness=0.5] (0,4.25)
        to [out=up, in=down, looseness=0.5] (0.7,4.625)
        to [out=up,in=down, looseness=0.5] (0,5)
        to [out=up, in=down] (0,6)
        to [out=up, in=left] (0.35,6.5)
        to [out=right, in=up] (0.7,6)
        to [out=down, in=up] (0.7,5)
        to [out=down, in=up, looseness=0.5] (0,4.625)
        to [out=down, in=up, looseness=0.5] (0.7,4.25)
        to [out=down, in=left, looseness=0.2] (5.2,3.9)
        to [out=right, in=up] (7,3.3)
        to [out=down, in=up] (7,3)
        to [out=down, in=right](5,2)
        to [out=left, in=up, looseness=0.4] (2.2,1.4)
        to [out=down, in=up, looseness=0.4] (4,0)
        to [out=right, in=left] (4.5,0)
        to [out=right, in=right] (4.5,5.5)
        to [out=left, in=right] (-0.5,5.5)
        to [out=left, in=left] (-0.5,0)
        to [out=right, in=left] (0,0);
        \flipcrossings{5,18,3,4,19,20,9,10,12,13,11,30,23,25,27,6,17}
    \end{knot}
\end{tikzpicture}
}
\caption*{$12n_{392}$}
\end{minipage}
\hfill
\begin{minipage}{0.24\textwidth}
\centering
\resizebox{\linewidth}{!}{
\begin{tikzpicture} 
\path[use as bounding box] (-3.6,0) rectangle (7.5,6.5);
    \begin{knot}[
            clip width=10,
            clip radius=8pt,
            consider self intersections=no splits,
	 end tolerance=0.1pt
            ]
        \strand [thick] (0,0)
        to [out=up,in=down] (0,6)
        to [out=up, in=left] (0.5,6.5)
        to [out=right, in=up] (1,6)
        to [out=down, in=up] (1,0)
        to [out=right, in=left] (3,0)
        to [out=up, in=right] (2,2.5)
        to [out=left, in=up, looseness=0.4] (-0.5,2.25)
        to [out=down, in=left, looseness=0.4] (2,2)
        to [out=right, in=down] (3,4)
        to [out=up, in=down, looseness=1.5] (4,6)
        to [out=up, in=right] (3.5,6.5)
        to [out=left, in=up] (3,6)
        to [out=down, in=up, looseness=1.5] (4,4)
        to [out=down, in=right] (2,1)
        to [out=left, in=left, looseness=5] (2,3.5)
        to [out=right, in=up] (4,0)
        to [out=right, in=left] (4.5,0)
        to [out=right, in=right] (4.5,5.5)
        to [out=left, in=right] (-0.5,5.5)
        to [out=left, in=left] (-0.5,0)
        to [out=right, in=left] (0,0);
        \flipcrossings{11, 12, 1, 6, 4, 9, 10, 14, 16}
    \end{knot}
\end{tikzpicture}
}
\caption*{$12n_{464}$}
\end{minipage}
\hfill
\begin{minipage}{0.24\textwidth}
\centering
\resizebox{\linewidth}{!}{
\begin{tikzpicture} 
\path[use as bounding box] (-3.6,0) rectangle (7.5,6.5);
    \begin{knot}[
            clip width=10,
            clip radius=8pt,
            consider self intersections=no splits,
	 end tolerance=0.1pt
            ]
        \strand [thick] (0,0)
        to [out=up, in=left, looseness=0.8] (2,1.2)
        to [out=right, in=down, looseness=0.4] (4.4,1.5)
        to [out=up,in=down, looseness=0.4] (4.4,4.3)
        to [out=up,in=right, looseness=0.4] (2,4.3)
        to [out=left,in=down, looseness=1] (0,6)
        to [out=up, in=left] (0.5,6.5)
        to [out=right, in=up] (1,6)
        to [out=down, in=left] (2,5)
        to [out=right, in=up, looseness=1] (5,4.3)
        to [out=down, in=up, looseness=0.4] (5,1.5)
        to [out=down, in=right, looseness=0.7] (2,0.5)
        to [out=left, in=up, looseness=0.7] (1,0)
        to [out=right, in=left] (3,0)
        to [out=up, in=left] (3.3,2.7)
        to [out=right, in=down] (7,3.5)
        to [out=up, in=right, looseness=0.5] (3.8,3.7)
        to [out=left, in=down] (4,6)
        to [out=up, in=right] (3.5,6.5)
        to [out=left, in=up] (3,6)
        to [out=down, in=left] (3.3,3.2)
        to [out=right, in=up] (7,2.5)
        to [out=down, in=right, looseness=0.5] (3.8,2.2)
        to [out=left, in=up] (4,0)
        to [out=right, in=left] (4.6,0)
        to [out=right, in=right] (4.6,5.5)
        to [out=left, in=right] (-0.5,5.5)
        to [out=left, in=left] (-0.5,0)
        to [out=right, in=left] (0,0);
        \flipcrossings{1, 2, 17, 18, 7, 8, 11, 12, 9, 23, 20, 25, 19}
    \end{knot}
\end{tikzpicture}
}
\caption*{$12n_{482}$}
\end{minipage}
\hfill
\begin{minipage}{0.24\textwidth}
\centering
\resizebox{\linewidth}{!}{
\begin{tikzpicture} 
\path[use as bounding box] (-3.6,0) rectangle (7.5,6.5);
    \begin{knot}[
            clip width=10,
            clip radius=8pt,
            consider self intersections=no splits,
	 end tolerance=0.1pt
            ]
       \strand [thick] (0,0)
        to [out=up,in=down] (0,6)
        to [out=up, in=left] (0.5,6.5)
        to [out=right, in=up] (1,6)
        to [out=down, in=up] (1,0)
        to [out=right, in=left] (3,0)
        to [out=up, in=right] (2,2.5)
        to [out=left, in=up, looseness=0.4] (-0.5,2.25)
        to [out=down, in=left, looseness=0.4] (2,2)
        to [out=right, in=down] (3,4)
        to [out=up, in=down, looseness=0.5] (4,4.5)
        to [out=up, in=down, looseness=0.5] (3,5)
        to [out=up, in=down, looseness=1] (3,6)
        to [out=up, in=left] (3.5,6.5)
        to [out=right, in=up] (4,6)
        to [out=down, in=up, looseness=1] (4,5)
        to [out=down, in=up, looseness=0.5] (3,4.5)
        to [out=down, in=up, looseness=0.5] (4,4)
        to [out=down, in=right] (2,1)
        to [out=left, in=left, looseness=5] (2,3.5)
        to [out=right, in=up] (4,0)
        to [out=right, in=left] (4.5,0)
        to [out=right, in=right] (4.5,5.5)
        to [out=left, in=right] (-0.5,5.5)
        to [out=left, in=left] (-0.5,0)
        to [out=right, in=left] (0,0);
        \flipcrossings{11, 12, 1, 6, 4, 9, 5, 14, 16}
    \end{knot}
\end{tikzpicture}
}
\caption*{$12n_{483}$}
\end{minipage}
\hfill
\begin{minipage}{0.24\textwidth}
\centering
\resizebox{\linewidth}{!}{
\begin{tikzpicture} 
\path[use as bounding box] (-3.6,0) rectangle (7.5,6.5);
    \begin{knot}[
            clip width=10,
            clip radius=8pt,
            consider self intersections=no splits,
	 end tolerance=0.1pt
            ]
        \strand [thick] (0,0)
        to [out=up,in=down] (0,6)
        to [out=up, in=left] (0.5,6.5)
        to [out=right, in=up] (1,6)
        to [out=down, in=up] (1,0)
        to [out=right, in=left] (3,0)
        to [out=up, in=right] (2,2.5)
        to [out=left, in=up, looseness=0.4] (-0.5,2.25)
        to [out=down, in=left, looseness=0.4] (2,2)
        to [out=right, in=down] (3,4)
        to [out=up, in=down, looseness=1.5] (4,6)
        to [out=up, in=right] (3.5,6.5)
        to [out=left, in=up] (3,6)
        to [out=down, in=up, looseness=1.5] (4,4)
        to [out=down, in=right] (2,1)
        to [out=left, in=left, looseness=5] (2,3.5)
        to [out=right, in=up] (4,0)
        to [out=right, in=left] (4.5,0)
        to [out=right, in=right] (4.5,5.5)
        to [out=left, in=right] (-0.5,5.5)
        to [out=left, in=left] (-0.5,0)
        to [out=right, in=left] (0,0);
        \flipcrossings{11, 12, 1, 6, 4, 9, 5, 14, 16}
    \end{knot}
\end{tikzpicture}
}
\caption*{$12n_{650}$}
\end{minipage}
\hfill
\begin{minipage}{0.24\textwidth}
\centering
\resizebox{\linewidth}{!}{
\begin{tikzpicture} 
\path[use as bounding box] (-3.6,0) rectangle (7.5,6.5);
    \begin{knot}[
            clip width=10,
            clip radius=8pt,
            consider self intersections=no splits,
	 end tolerance=0.1pt
            ]
        \strand [thick] (1.5,0)
        to [out=right, in=down] (2,0.5)
        to [out=up, in=down] (2,2.5)
        to [out=up, in=left] (2.5,3)
        to [out=right, in=up] (3,2.5)
        to [out=down, in=up] (3,0.5)
        to [out=down, in=left] (4.5,0)
        to [out=right, in=down] (6,0.5)
        to [out=up, in=down] (6,5)
        to [out=up, in=right] (3,6)
        to [out=left, in=up] (0,5)
        to [out=down, in=up] (0,1.5)
        to [out=down, in=left] (0.5,1)
        to [out=right, in=left] (5,1)
        to [out=right, in=down] (7,2.5)
        to [out=up, in=right] (5,4)
        to [out=left, in=right] (-0.5,4)
        to [out=left, in=up] (-1,3)
        to [out=down, in=left] (-0.5,2)
        to [out=right, in=left] (3.5,2)
        to [out=right, in=down] (4,3)
        to [out=up, in=right] (2.5,5)
        to [out=left, in=up] (1,3)
        to [out=down, in=up] (1,0.5)
        to [out=down, in=left] (1.5,0);
        \flipcrossings{2,5,8,10,12}
    \end{knot}
    \begin{knot}[
            draft mode=strands,
            clip width=10,
            clip radius=8pt,
            consider self intersections=false,
	        end tolerance=0.1pt
            ]
        \strand[thick,red](1.7,2)
        to [out=up, in=left] (2,2.3)
        to [out=right, in=up] (2.3,2)
        to [out=down, in=right] (2,1.7)
        to [out=left, in=down] (1.7,2);
    \end{knot}
    \begin{knot}[
            draft mode=strands,
            clip width=10,
            clip radius=8pt,
            consider self intersections=false,
	        end tolerance=0.1pt
            ]
        \strand[thick,red](0.9,4)
        to [out=up, in=left] (1.2,4.3)
        to [out=right, in=up] (1.5,4)
        to [out=down, in=right] (1.2,3.7)
        to [out=left, in=down] (0.9,4);
    \end{knot}
\end{tikzpicture}
}
\caption*{$12n_{831}$}
\end{minipage}

\end{centering}

\captionof{figure}{$2$-adjacent knots with $\leq12$ crossings. For each finger-move diagram, one crossing from each hook makes up the $2$-adjacency set. Those knots for which a finger-move diagram was unable to be found are still included, and their $2$-adjacency set is indicated with circles.}
\label{fig:2adj-knots}

\end{figure}

{\begin{figure}
    \centering
    \includegraphics[width=.5\linewidth]{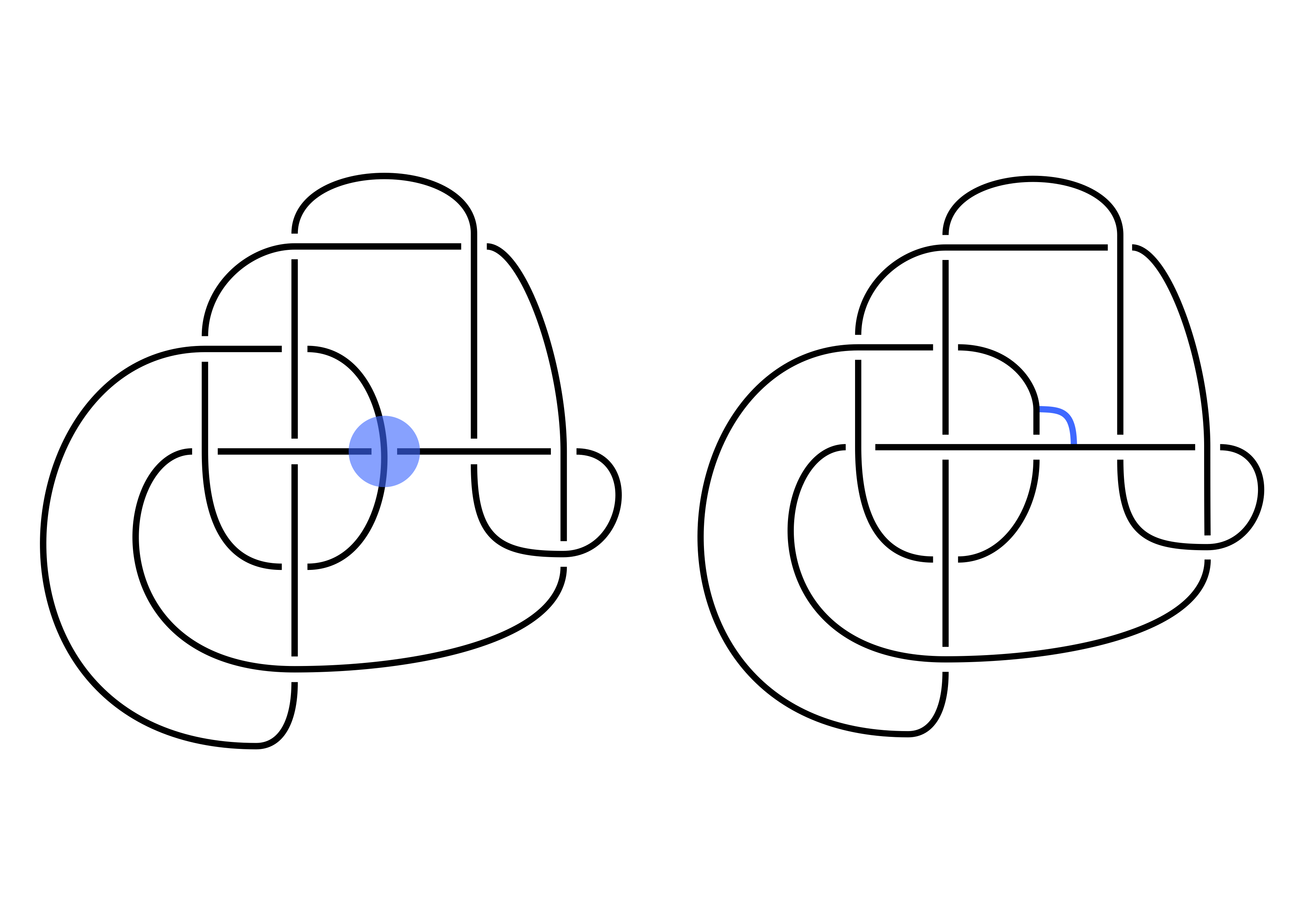} \includegraphics[width=.45\linewidth]{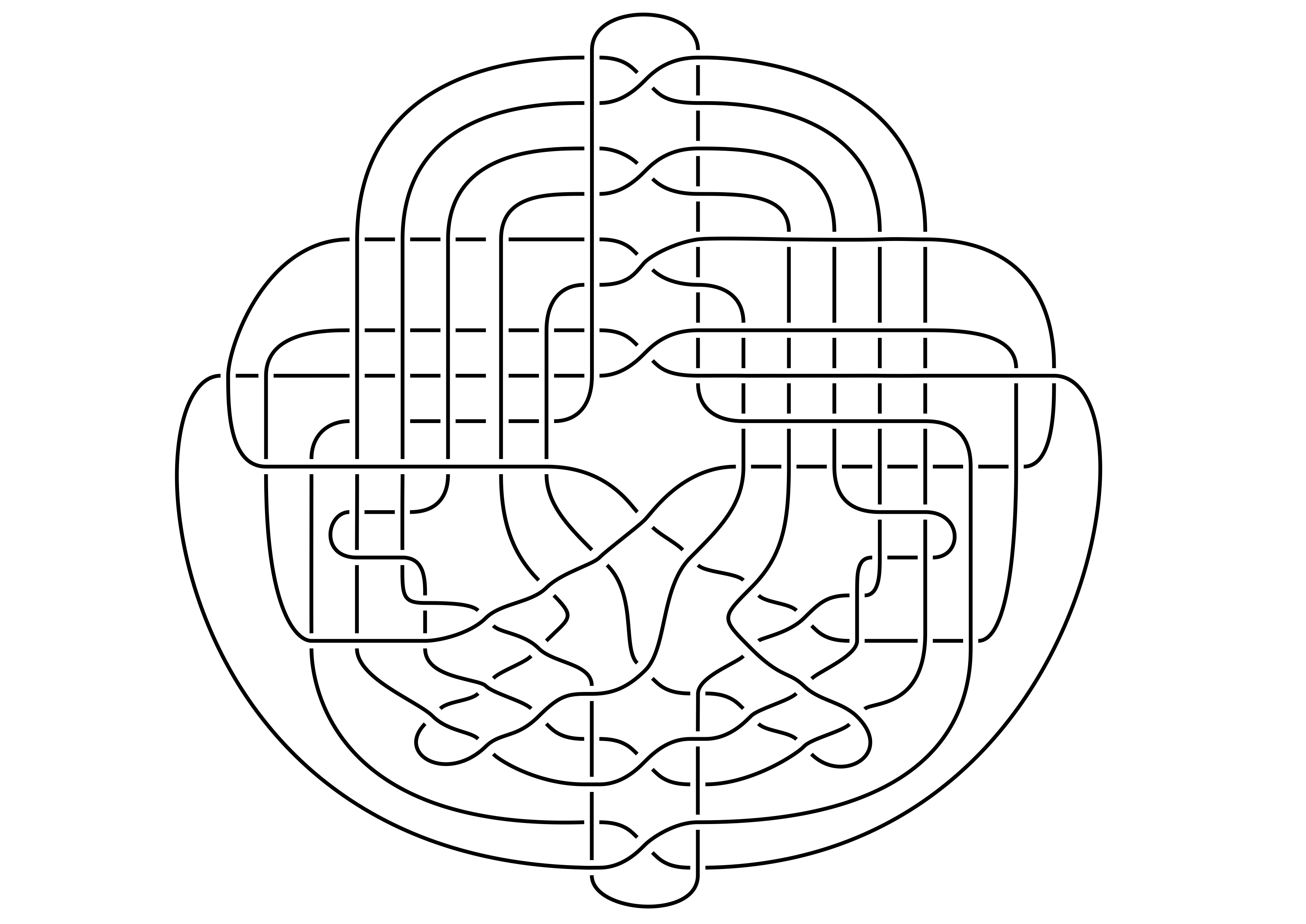}
    \includegraphics[width=.5\linewidth]{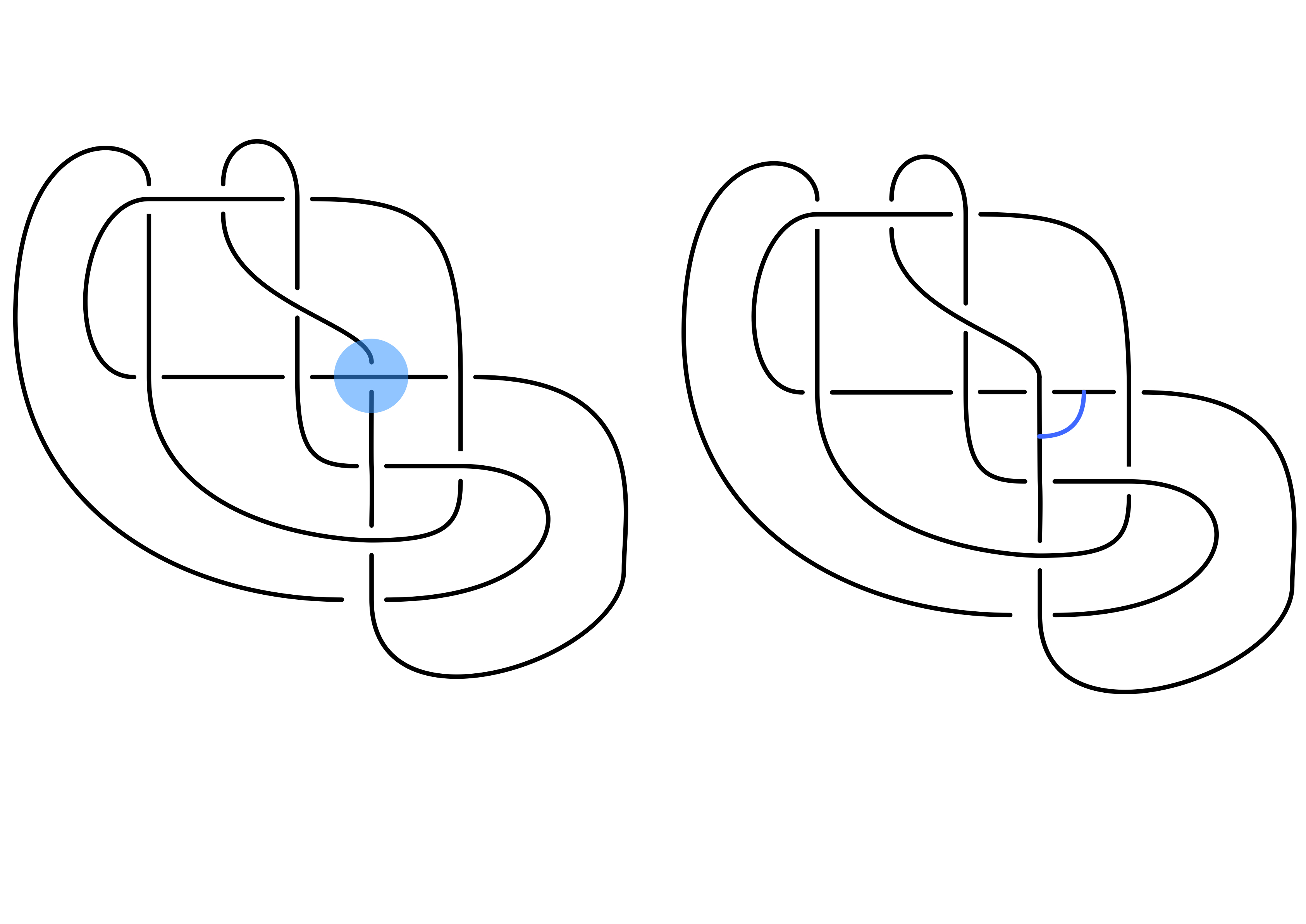} \includegraphics[width=.45\linewidth]{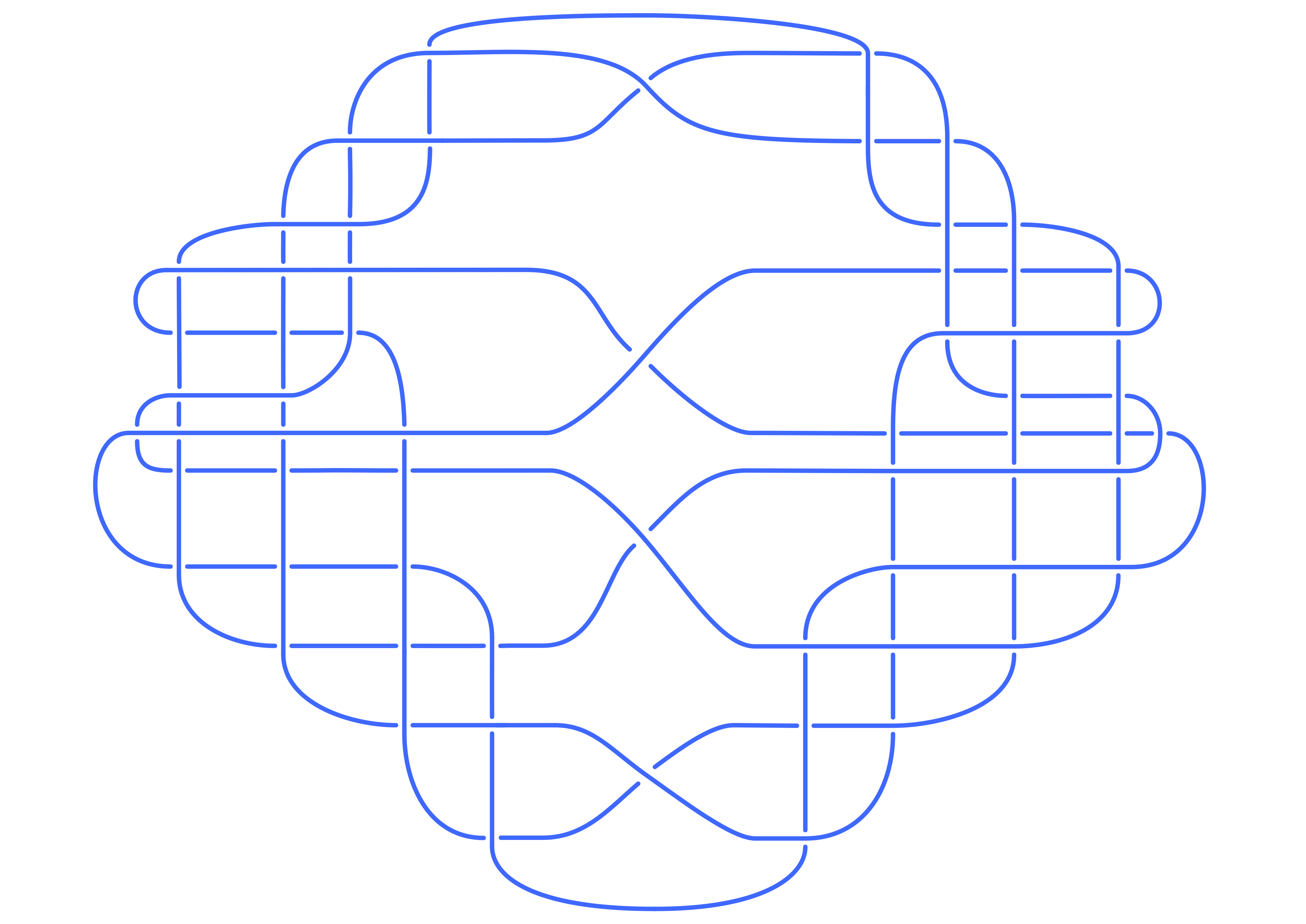}
    \includegraphics[width=.5\linewidth]{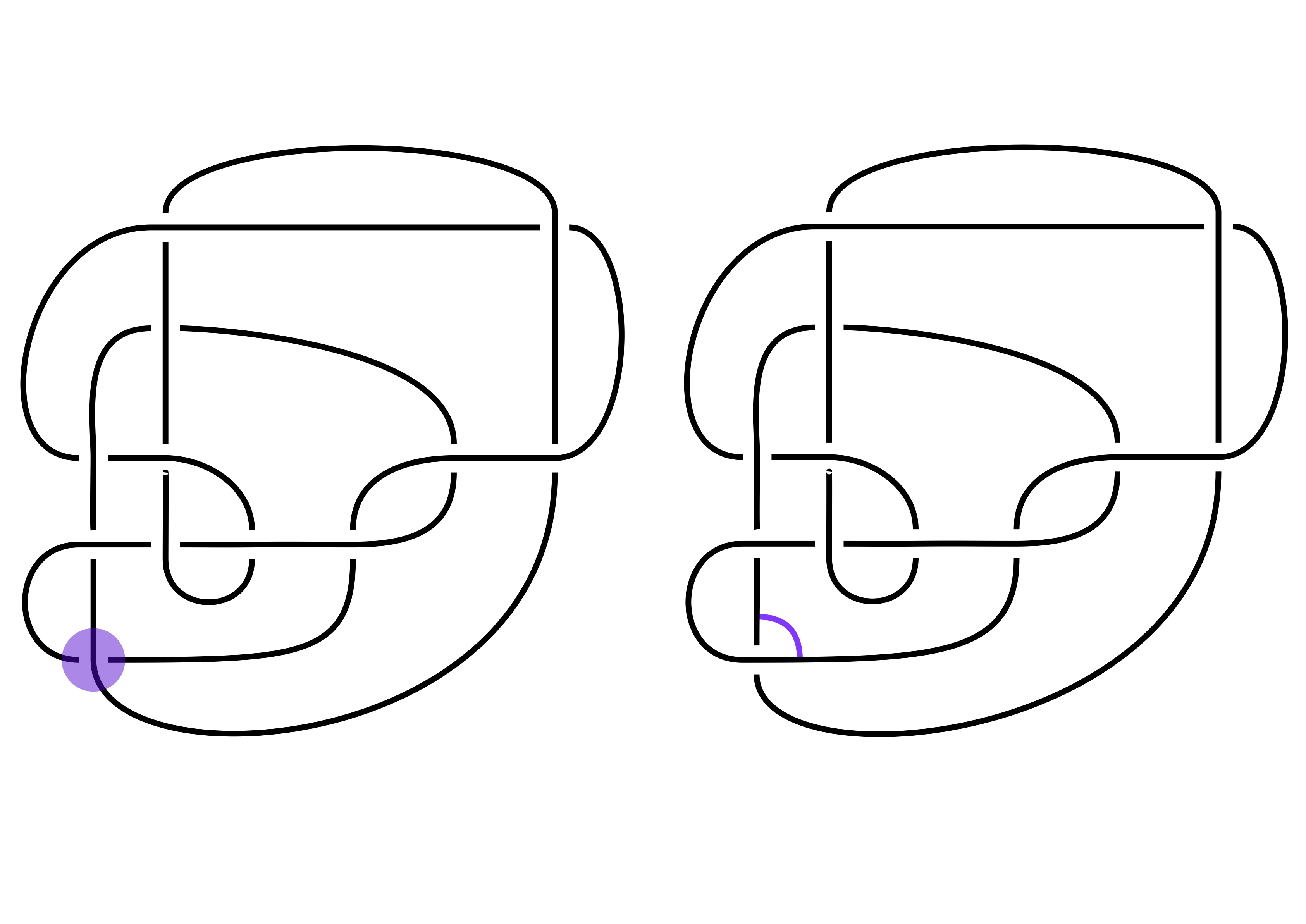} \includegraphics[width=.45\linewidth]{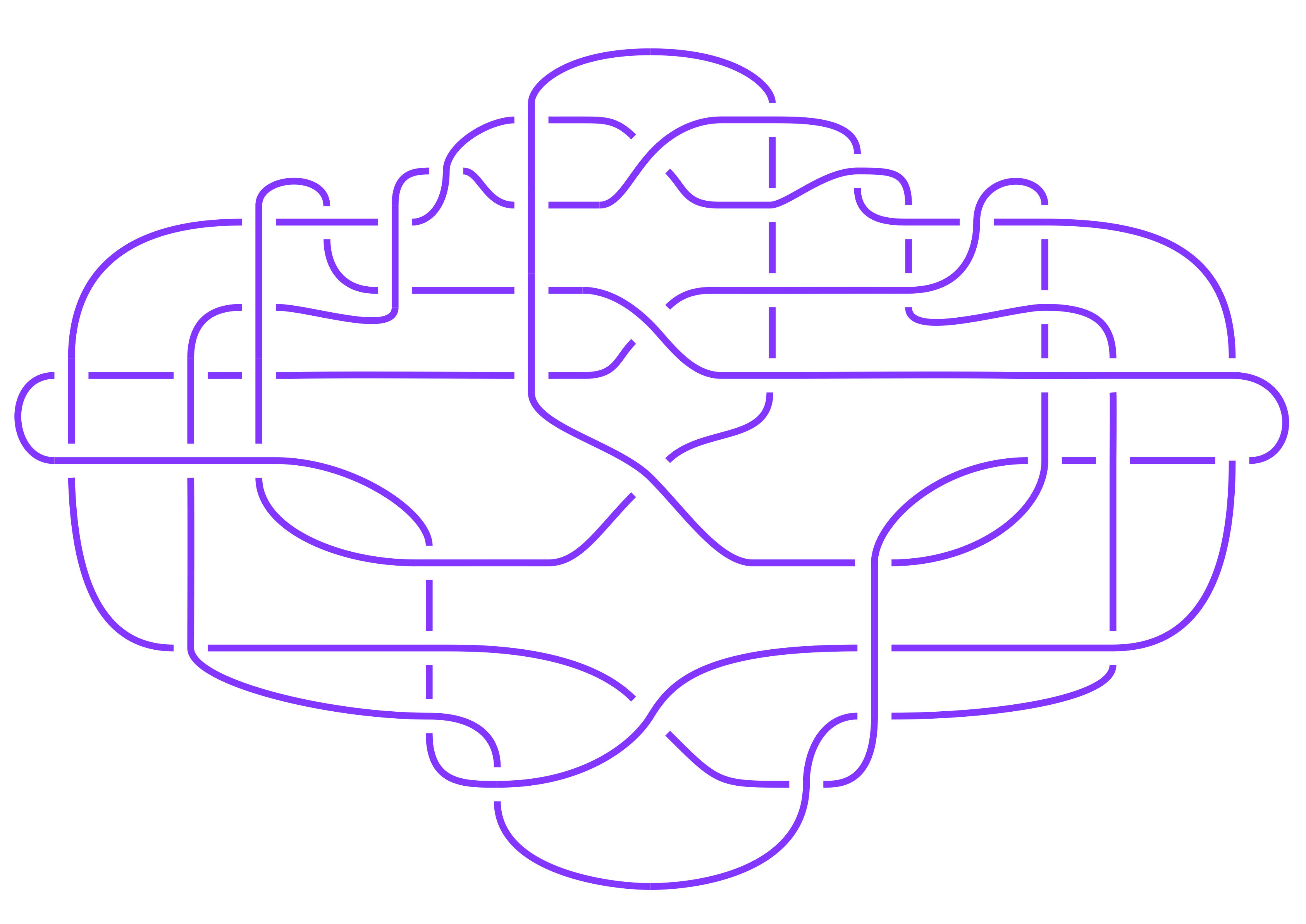}
    \caption{Finding the lifted arc $J$ for the knots $12a_{358}$, $12n_{620}$, and $12n_{656}$. The diagrams on the left shows the knot with an unknotting crossing circled. In the middle diagrams, we have changed the unknotting crossing and added an arc which lifts to the figures on the right.}
    \label{fig.12a_358-crossing}
\end{figure}}

{\begin{table} 
        \centering
        \begin{tabular}{|l|r|r|r|} \hline 
            Knot $K$ & Signature & Conway polynomial  & Determinant \\ \hline 
            $13n_{589}$ & $-2$ & $z^6+1$ & $63$ \\ \hline 
            $13n_{1179}$ & $0$ & $z^8+2z^6+z^4+1$ & $145$ \\ \hline 
            $13n_{1202}$ & $0$ & $z^8+2z^6+z^4+1$ & $145$ \\ \hline 
            $13n_{1486}$ & $0$ & $z^8+3z^6+2z^4-z^2+1$ & $101$ \\ \hline 
            $13n_{1822}$ & $-2$ & $-z^6-4z^4+z^2+1$ & $3$ \\ \hline 
            $13n_{2073}$ & $2$ & $z^6 + 1$ & $63$ \\ \hline 
            $13n_{2278}$ & $2$ & $-2z^4 + z^2 + 1$ & $35$ \\ \hline 
            $13n_{2337}$ & $-2$ & $-z^4 + 1$ & $15$ \\ \hline 
            $13n_{2693}$ & $-2$ & $-z^8 - 2z^6 - z^4 + 1$ & $143$ \\ \hline 
            $13n_{2724}$ & $2$ & $2z^6 - z^4 + 1$ & $143$ \\ \hline 
            $13n_{3017}$ & $-2$ & $-z^4 + 1$ & $15$ \\ \hline 
            $13n_{3416}$ & $0$ & $z^4 + 1$ & $17$ \\ \hline 
            $13n_{4913}$ & $-2$ & $-z^8 - 2z^6 - z^4 + 1$ & $143$ \\ \hline 
        \end{tabular}
        \caption{Table of known $2$-adjacent knots with $13$ crossings. Each has an easily found $2$-adjacency set.}
        \label{tab:13crossing2adjknots}
    \end{table}}

    \begin{figure} 
    \flushleft
        $\left\{13a_{1328}, 13a_{2671}, 13a_{3150}, 13a_{3634}, 13n_{137}, 13n_{167}, 13n_{179}, 13n_{372}, 13n_{375}, 13n_{422}, 13n_{423}, 13n_{630}, \right.$ \\ $\left. \quad 13n_{904}, 13n_{940}, 13n_{1509}, 13n_{1513}, 13n_{1572}, 13n_{1690}, 13n_{1861}, 13n_{1923}, 13n_{2012}, 13n_{2057}, 13n_{2085}, \right.$ \\ $\left. \quad 13n_{2251}, 13n_{2426}, 13n_{2427}, 13n_{2522}, 13n_{2696}, 13n_{2734}, 13n_{2792}, 13n_{2828}, 13n_{2834}, 13n_{2865}, 13n_{2868}, \right.$ \\ $\left. \quad 13n_{2889}, 13n_{2933}, 13n_{2956}, 13n_{2997}, 13n_{3072}, 13n_{3218}, 13n_{3299}, 13n_{3485}, 13n_{3563}, 13n_{3574}, 13n_{3669}, \right.$ \\ $\left. \quad 13n_{3796}, 13n_{3888}, 13n_{3895}, 13n_{3950}, 13n_{3984}, 13n_{4111}, 13n_{4113}, 13n_{4342}, 13n_{4390}, 13n_{4414}, 13n_{4420}, \right.$ \\ $\left. \quad 13n_{4430}, 13n_{4465}, 13n_{4523}, 13n_{4542}, 13n_{4582}, 13n_{4591}, 13n_{4650}, 13n_{4729}, 13n_{4765}, 13n_{4788}, 13n_{4874}, \right.$ \\ $\left. \quad 13n_{4914}, 13n_{4955}, 13n_{4957}, 13n_{4982}, 13n_{5062}, 13n_{5070}, 13n_{5084}\right\}$
        \caption{List of $13$ crossing knots whose $2$-adjacency status is not resolved by Theorem \ref{thm-newobst}, Proposition \ref{prop-basic-obstructions}, Proposition \ref{tao5.4}, Proposition \ref{tao5.2}, or Proposition \ref{taoconway}. The method described in Section \ref{rulingoutalexpolys} was not used for this list.}
        \label{13crossingunresolvedknots}
    \end{figure} 

\noindent

\clearpage
\bibliographystyle{alpha}
\bibliography{references}

\end{document}